\documentclass[11pt]{article}
\usepackage{amsmath,amsfonts,amsthm,amssymb,subfiles}

\usepackage{fullpage, hyperref}

\usepackage{tikz,amstext}

\tikzstyle{symb} = [draw, line width=10pt]

\newcommand*{\recsumbase}{%
\begin{tikzpicture}
  \draw[symb] (0,1) rectangle (6,5);
  \draw[symb] (0,3) -- (6,3) (1.75,1) -- (1.75,5) (4.25,1) -- (4.25,5);
  \end{tikzpicture}
}

\newcommand*{\sqsumbase}{%
\begin{tikzpicture}
  \draw[symb] (0,1) rectangle (4,5);
  \draw[symb] (0,3) -- (4,3) (2,1) -- (2,5);
\end{tikzpicture}
}

\newcommand*{\sqmultbase}{%
\begin{tikzpicture}
  \draw[symb] (0,1) rectangle (4,5);
  \draw[symb] (0,1) -- (4,5) (0,5) -- (4,1);
\end{tikzpicture}
}

\def\withsub_#1{_{#1}\!\egroup}

\newcommand*{\sqmultsmall}{\,\resizebox{0.5em}{0.5em}{\sqmultbase}}
\newcommand*{\sqsumsmall}{\,\resizebox{0.5em}{0.5em}{\sqsumbase}}
\newcommand*{\recsumsmall}{\,\resizebox{0.75em}{0.5em}{\recsumbase}}

\newcommand*{\sqmultlarge}{\resizebox{0.625em}{0.625em}{\sqmultbase}}
\newcommand*{\sqsumlarge}{\resizebox{0.625em}{0.625em}{\sqsumbase}}
\newcommand*{\recsumlarge}{\resizebox{0.875em}{0.625em}{\recsumbase}}

\makeatletter
\newcommand*{\sqmult}{\mathbin \bgroup \mathchoice{\sqmultlarge}{\sqmultlarge}{\sqmultsmall}{\sqmultsmall}\@ifnextchar_{\withsub}{\egroup}}
\newcommand*{\sqsum}{\mathbin \bgroup \mathchoice{\sqsumlarge}{\sqsumlarge}{\sqsumsmall}{\sqsumsmall}\@ifnextchar_{\withsub}{\egroup}}
\newcommand*{\recsum}{\mathbin \bgroup \mathchoice{\recsumlarge}{\recsumlarge}{\recsumsmall}{\recsumsmall}\@ifnextchar_{\withsub}{\egroup}}
\makeatother

\newcommand{\cauchy}[2]{\mathcal{G}_{#1} \left(#2 \right)}
\newcommand{\invcauchy}[2]{\mathcal{K}_{#1} \left(#2 \right)}
\newcommand{\rtrans}[2]{\mathcal{R}_{#1} \left(#2 \right)}
\newcommand{\myinvcauchy}[3]{\mathcal{K}^{#1}_{#2} \left( #3 \right)}
\newcommand{\myrtrans}[3]{\mathcal{R}^{#1}_{#2} \left( #3 \right)}

\newcommand{\strans}[2]{\widetilde{\mathcal{S}}_{#1} \left(#2 \right)}

\newcommand{\mtrans}[2]{\widetilde{\mathcal{M}}_{#1} \left(#2 \right)}
\newcommand{\invmtrans}[2]{\widetilde{\mathcal{M}}^{(-1)}_{#1} \left(#2 \right)}

\newcommand{\mystrans}[3]{\widetilde{\mathcal{S}}^{#1}_{#2} \left(#3 \right)}
\newcommand{\myinvmtrans}[3]{\widetilde{\mathcal{N}}^{#1}_{#2} \left(#3 \right)}

\newtheorem{theorem}{Theorem}[section]
\newtheorem{corollary}[theorem]{Corollary}
\newtheorem{lemma}[theorem]{Lemma}

\theoremstyle{definition}
\newtheorem{definition}[theorem]{Definition}
\newtheorem{remark}[theorem]{Remark}
\newtheorem{example}[theorem]{Example}

\newcommand{\R}{\mathbb{R}}
\newcommand{\C}{\mathbb{C}}

\newcommand{\E}[2]{\,\mathbb{E}_{#1}\! \left\{ {#2} \right\}}

\newcommand{\mydet}[1]{\det \left[ #1 \right]}
\newcommand{\D}[1]{\mathrm{D} \left( #1 \right)}
\newcommand{\Tr}[1]{\mathrm{Tr} \left[ #1 \right]}
\newcommand{\tr}[1]{\mathrm{tr} \left[ #1 \right]}
\newcommand{\myphi}[1]{\mathrm{\phi} \left[ #1 \right]}

\newcommand{\ifno}[2]{
\makeatletter
\@ifundefined{#1}{ #2 }{}
}

\newcommand{\AND}{\qquad\text{and}\qquad}

\renewcommand{\d}[1]{\,\mathrm{d}{#1}}
\renewcommand{\L}[1]{\mathcal{L} \left\{  #1 \right\}}

\newcommand{\deriv}[1]{ \frac{\partial}{\partial #1}}
\newcommand{\ideriv}[2]{ \frac{\partial^{#2}}{(\partial #1)^{#2}}}

\renewcommand{\invcauchy}[2]{ \mathcal{G}_{#1}^{-1} \left( #2 \right)}

\newcommand{\fkdet}[1]{\Delta^+ \left( #1 \right)}

\newcommand{\iu}{i}

\newcommand{\M}{m}
\newcommand{\marcus}{U}

\newcommand{\mynorm}[2]{ \left\| #1 \right\|_{#2}}
\newcommand{\infnorm}[1]{\mynorm{#1}{\infty}}

\newcommand{\eigen}[1]{\lambda \left( #1 \right)}

\newcommand{\mymod}[2]{\mod [{#1}^{#2} ] }

\renewcommand{\M}{d}
\renewcommand{\sqsum}{\boxplus}
\renewcommand{\sqmult}{\boxtimes}
\begin{document}

\title{
Polynomial convolutions and (finite) free probability
}
\author{Adam W. Marcus\thanks{Research supported by NSF CAREER Grant 
DMS-1552520.}
\\ \'Ecole Polytechnique F\'ed\'erale de Lausanne
}

\maketitle

\begin{abstract}
We introduce a finite version of free probability and show the link between recent results using polynomial convolutions and the traditional theory of free probability. 
One tool for accomplishing this is a seemingly new transformation that allows one to reduce computations in our new theory to computations using classically independent random variables.
We then explore the idea of finite freeness and its implications.
Lastly, we show applications of the new theory by deriving the finite versions of some well-known free distributions and then proving their associated limit laws directly.
In the process, we gain a number of insights into the behavior of convolutions in traditional free probability that seem to get lost when the operators being convolved are no longer finite.
\end{abstract}

\section{Introduction} \label{sec:intro}

After its introduction in 1986 in a series of papers by Dan Voiculescu, free probability has seen an incredible growth in both its theory and its applications.
This has included a theory of {\em free cumulants}, first introduced by Nica 
and Speicher, which gave a unified framework for understanding classical and 
free independence through the lens of combinatorics \cite{nica2006lectures}.
It has been used as a tool in a variety of areas, including random matrix theory, combinatorics, representations of symmetric groups, large deviations, and quantum information theory.
In most cases, the relationships mentioned above only exist in an asymptotic sense, primarily due to the fact that no nontrivial free objects exist in finite dimensions.
However, recent work of the author with Daniel Spielman and Nikhil Srivastava \cite{ICM, conv, IF4} suggests that the behavior of finite structures closely resembles the asymptotic ``free'' behavior, despite not technically being ``free''. 
The goal of this paper is to introduce a theory that we call ``finite free 
probability'' as a way to extend the fundamental concepts and insights of free 
probability to finite objects using polynomial convolutions.

\subsection{A brief introduction to free probability}
\label{sec:free_prob}

We begin with an informal introduction to free probability (in particular, as it relates to classical probability).
Let $(M_1, \mu_1)$ and $(M_2, \mu_2)$ be probability spaces and let $p(x, y)$ a bivariate polynomial.
In classical probability, one realizes $\mu_1$ and $\mu_2$ as {\em random variables} $X_1$ and $X_2$ with the goal of investigating the distribution $\mu_Y$ of objects of the form $Y = p(X_1, X_2)$ living in the tensor product $M_1 \otimes M_2$.
Each of these probability spaces are equipped with a {\em test function} $\E{}{}$ that allows one to measure various functions of the $X_i$.
Typically, $\mu_Y$ cannot be calculated without knowledge of the joint probability distribution $\mu_{X_1, X_2}$.
However, there is one ``special'' joint distribution for which one can calculate $\mu_Y$ only knowing $\mu_1$ and $\mu_2$ for any $p$ ---  the case when $X_1$ and $X_2$ are {\em independent}.

In noncommutative probability, one realizes $\mu_1$ and $\mu_2$ as the spectrum of {\em random operators} $A_1$ and $A_2$.
The goal is still to investigate the distribution $\mu_Y$ of objects of the form $Y = p(A_1, A_2)$, however this time such objects live in the free product of $M_1$ and $M_2$.
Free independence (or ``freeness'') is the free product analogue of classical independence on tensor products.
In particular, it is the ``special'' joint distribution that allows $\mu_Y$ to be calculated completely in terms of $\mu_1$ and $\mu_2$.
This time the spaces are equipped with a test function $\myphi{}$ which has tracial properties (since the distribution lives on the spectrum of operators) that is used to measure various functions of the $A_i$.

Formally, we say that two random operators $A, B$ are freely independent if for all $n$ and all polynomials $p_1, \dots, p_{2n}$, we have that 
\[
\myphi{p_1(A) p_2(B) \dots p_{2n-1}(A) p_{2n}(B)} 
= 0
\]
whenever $\myphi{p_{2i-1}(A)} = \myphi{p_{2i}(B)} = 0$ for all $1 \leq i \leq n$.
In practice, the main difference between free and classical independence is that free independence respects the noncommutativity of random operators.
For example, when $X$ and $Y$ are independent, we have
\[
\E{}{X^2Y^2} 
= \E{}{XYXY} 
= \E{}{X^2}\E{}{Y^2}
\]
whereas when $A$ and $B$ are freely independent, we have
\[
\myphi{A^2B^2} 
= \myphi{A^2}\myphi{B^2}
\]
while 
\[
\myphi{ABAB} 
= \myphi{A^2}\myphi{B}^2 + \myphi{A}^2\myphi{B^2} - \myphi{A}^2\myphi{B}^2
\]

On the other hand, there are a number of connections between classical and free independence, with varying degrees of understanding.
Random matrix theory, in particular, captures a number of these connections, as many times classical independence of individual entries leads to free independence of spectral distributions.
For example, we have the following theorem of Voiculescu:
\begin{theorem}
\label{thm:free}
Let $\mu_A$ and $\mu_B$ be discrete probability distributions and let $A_{\M}$ and $B_{\M}$ be $\M \times \M$ real symmetric matrices with eigenvalue distribution $\mu_A$ and $\mu_B$ (respectively).
Let $R_{\M}$ and $R'_{\M}$ be i.i.d. random matrices drawn uniformly (via the Haar measure) from $O(\M)$ (the orthogonal group).
Then the operators 
\[
A 
= \lim_{\M \to \infty} R_{\M} A_{\M} R^T_{\M}
\AND
B 
= \lim_{\M \to \infty} R'_{\M} B_{\M} R'^T_{\M}
\]
are freely independent.
\end{theorem}
Theorem~\ref{thm:free} reveals one of the hurdles in applying free probability theory to finite structures: finite dimensional matrices are freely independent if and only if one of them is a multiple of the identity.
As a result, applications of the theory must be done in an asymptotic sense.

\subsection{Convolutions}
\label{sec:free_convolutions}

The computational aspects of free probability center around computing polynomial functions of random operators.
Unlike in the classical case, however, even the basic operations of addition and multiplication are nontrivial to compute, even when the operators involved are freely independent.

\subsubsection{Additive Convolution}

Let $A$ and $B$ be operators with spectral distributions $\mu_A$ and $\mu_B$.
The {\em free additive convolution} of $\mu_A$ and $\mu_B$ (written $\mu_{A} \boxplus \mu_{B}$) is the distribution of the operator $A + B$ in the case that $A$ and $B$ are freely independent.
To compute such a thing, we begin by computing the {\em Cauchy transform} of $\mu_A$
\begin{equation}
\label{eq:cauchy}
\cauchy{\mu_A}{x} 
= \sum_i \frac{1}{x^{i+1}} \myphi{A^i} 
= \myphi{(xI - A)^{-1}} 
= \int \frac{\mu_A(t)}{x-t} \d{t}
\end{equation}
and then the so-called {\em R-transform} as a function of the inverse of $\cauchy{\mu_A}{x}$
\[
\rtrans{\mu_A}{x} 
= \invcauchy{\mu_A}{x} - \frac{1}{x} 
= \invcauchy{\mu_A}{x} - \invcauchy{\mu_0}{x}
\]
where $0$ is the zero matrix.
Here, ``inverse'' means the compositional inverse of the power series of $\cauchy{\mu_A}{x}$ expanded around $x = \infty$.
When $A$ and $B$ are freely independent, one then has
\begin{equation}
\label{eq:Rtransform}
\rtrans{\mu_A \boxplus \mu_B}{x} = \rtrans{\mu_A}{x} + \rtrans{\mu_B}{x}.
\end{equation}
Alternatively, one could simply define the free additive convolution $\mu_{A} \boxplus \mu_{B}$ as being the spectral distribution satisfying Equation~(\ref{eq:Rtransform}).
In the case that all of the operators have compact support, this is known to uniquely define the distribution $\mu_{A} \boxplus \mu_{B}$ \cite{voiculescu}.

We remark that Equation~(\ref{eq:cauchy}) shows the relationship between the Cauchy transform and the (non-exponential)  moment generating function
\[
\frac{1}{x}\cauchy{\mu}{\frac{1}{x}}
= \sum_i M_i(\mu_A)x^i
= \sum_i \myphi{A^i}x^i
\]
where 
\[
M_i 
= \int x^i \mu_A(x) \d{x}
\]
is the $i$th moment.
In this respect, it is similar to the way in which one would form the classical additive convolution by forming the moment generating function $\E{}{e^{xY}}$ and then adding some transformation of it (in this case, $\ln \E{}{e^{xY}}$).

\subsubsection{Multiplicative Convolution}

For the multiplicative case, one uses a variant of the Cauchy transform above
\[
\mtrans{\mu_A}{x} 
= x\cauchy{\mu_A}{x} - 1
\]
and then form (what we call) the {\em modified S-transform} by inverting a power series (this time taken at $x = 0$)
\[
\strans{\mu_A}{x} 
= \frac{s}{1+s}\invmtrans{\mu_A}{x} 
= \frac{\invmtrans{\mu_A}{x}}{\invmtrans{\mu_I}{x}}
\]
where $I$ is the identity matrix.
When $A$ and $B$ are freely independent, the moments of $\mu_{A} \boxtimes \mu_{B}$ are then defined by 
\begin{equation}
\label{eq:Stransform}
\strans{\mu_A \boxtimes \mu_{B}}{x} 
= \strans{\mu_A}{x} \strans{\mu_B}{x}
\end{equation}
\begin{remark}
We call this the ``modified'' S-transform because it does not use the same definition that is typically seen in free probability \cite{voiculescu}.
The relation to the usual definition is simply
\[
\strans{\mu_A}{x} 
= \frac{1}{\mathcal{S}_{\mu_A}(x)}.
\]
This of course does not change the relationship in Equation~(\ref{eq:Stransform}) in any way, and so for the purposes of traditional free probability it seems like a silly alteration.
However, in the finite setting, it does seem to be the more appropriate choice (as is discussed in \cite{conv}).
\end{remark}

\subsection{Previous Work}

The idea of extending free probability to finite matrices is not new.
As best the author could tell, it was first proposed by Edelman and Rao in 2005 \cite{orig_edelman}.
Later, the same authors suggested an implementation of such a theory, although with a focus on the computational (i.e. with Matlab) aspects of free probability \cite{edelman_polynomials}.
Their implementation uses polynomials (just as ours will); however, their polynomials are substantially different from the ones used in this paper.
Rather than encoding distributions in the roots of a polynomial, they encode various transforms as the solutions of bivariate polynomials and then give operations on these polynomials that capture the behavior of additive and multiplicative convolution.

The use of polynomials in \cite{edelman_polynomials} is far from coincidental.
Polynomials have been closely connected to random matrix theory from its initial beginnings.
It was recognized early on that the spectral distributions of random matrices matched the asymptotic root distribution of various orthogonal polynomials, and many of the known results regarding spectral distributions have been proved using this correspondence \cite{deift_book}.
Polynomials have also been used in similar ways in a strictly free probabilistic setting \cite{anshelevich}, and are one of the major tools in the investigation of {\em universality} for random matrices \cite{universality}.

Lastly, we mention the work of Pereira connecting so-called {\em trace vectors} to hyperbolic polynomials \cite{trace_vectors_polynomials}, which can be seen as a special case of the finite freeness developed in Section~\ref{sec:ff_position}.
See Section~\ref{sec:ff_more} for more details.
 
\subsection{New Results}

The contribution of this paper is to show a direct link between polynomial 
convolutions and free probability that the author believes is fundamentally new.
Previous work used polynomials only for their asymptotic properties, and 
typically focused on orthogonal polynomials.
In this work, we show that polynomials actually exhibit a close link to free probability on a finite scale.
Furthermore, this will provide a framework for understanding recent work in topics such as restricted invertibility \cite{ICM} and Ramanujan graphs \cite{IF4}.
By linking these methods, we hope to provide a systematic way to use free probability as a tool in combinatorics and graph theory (and vice versa).

\subsection{Organization}

We begin in Section~\ref{sec:prelims} by listing the collection of ideas and definitions that we will use. 
In Section~\ref{sec:marcus_transform}, we will introduce a transformation on finite multisets that we call the \marcus~transform, which will be useful in the computations of Section~\ref{sec:convolutions}.
In Section~\ref{sec:convolutions}, we show how the symmetric additive and symmetric multiplicative polynomial convolutions introduced in \cite{conv} relate to convolutions from free probability.
In Section~\ref{sec:ff_position}, we then examine the properties of finite 
freeness and use them to derive a collection of majorization inequalities.
Lastly, in Section~\ref{sec:applications}, we show some applications of the theory by deriving the finite versions of various laws in free probability and then deriving their associated limit theorems directly.
We also show how one the theory can be used to shine a more intuitive light on results such as Bourgain and Tzafriri's restricted invertibility theorem \cite{BT, ICM}.
We finish with some concluding remarks, suggestions for further research directions, and acknowledgements.

\section{Preliminaries} \label{sec:prelims}

Throughout the paper, we will focus on real symmetric matrices (although the 
results can easily be extended).
We will write $\rho_A$ to denote the largest eigenvalue of 
a matrix $A$ and $\tr{}$ to denote the normalized trace (so 
that $\tr{I} = 1$).
In the case of matrices, we will use $\Tr{}$ to denote the usual trace (so that $\Tr{I} = \dim(I)$).

The first part of this section will review the polynomial convolutions from \cite{conv} that will be the focus of this paper.
The remainder of the section will be used to introduce some of the tools that will be used in the analysis.

\subsection{Polynomial convolutions}
\label{sec:formulas}
Finite convolutions of polynomials were introduced in \cite{conv}.
We will be concerned with the real, symmetric cases.
Let $A$ and $B$ be $\M \times \M$ real symmetric matrices with 
\[
p(x) 
= \mydet{xI - A}
\AND
q(x) 
= \mydet{xI - B}.
\]
\begin{definition}
The {\em symmetric additive convolution} of $p$ and $q$ is defined as
\[
[p \sqsum_{\M} q](x) 
= \E{Q}{\mydet{xI - A - QBQ^T}}
\]
where the expectation is taken over orthonormal matrices $Q$ distributed uniformly (via the Haar measure).
For $A$ and $B$ positive definite, the {\em symmetric multiplicative convolution} of $p$ and $q$ is defined as
\[
[p \sqmult_{\M} q](x) 
= \E{Q}{\mydet{xI - AQBQ^T}}.
\]
\end{definition}
Furthermore, they note that when $p$ and $q$ each have all real roots, then both 
\[
[p \sqsum_{\M} q](x)
\AND
[p \sqmult_{\M} q](x)
\]
have all real roots, due to the (much more general) general theory developed in \cite{BB2}. 
The following linear formulas were proved as well:
\begin{theorem}
If
\[
p(x) 
= \sum_i x^{\M-i} (-1)^i p_i 
\AND
q(x) 
= \sum_i x^{\M-i} (-1)^i q_i 
\]
then
\begin{equation}
\label{eq:add_conv}
[p \sqsum_{\M} q](x) 
= \sum_{i + j \leq \M} x^{\M-i-j}(-1)^{i+j} \frac{(\M-i)!(\M-j)!}{(\M-i-j)!\M!} p_i q_j
\end{equation}
and
\begin{equation}
\label{eq:mult_conv}
[p \sqmult_{\M} q](x) 
= \sum_{i = 0}x^{\M-i-j}(-1)^i \frac{p_iq_i}{\binom{\M}{i}}
\end{equation}
\end{theorem}

\renewcommand{\deriv}[1]{ \frac{\partial}{\partial #1}}

\subsection{Laplace Transform}\label{sec:Laplace}
For a function $f$, the {\em Laplace transform} is defined as 
\[
\L{f}(s) 
= \int_0^{\infty} e^{-xs} f(x) \d{x}.
\]
The Laplace transform is quite useful combinatorially due to its ability to turn exponential generating functions into ordinary generating functions (and vice versa), as 
\begin{equation}
\label{eq:laplace_k}
\L{x^k}(s) 
= \frac{k!}{s^{k+1}}.
\end{equation}

Those uninterested in the details of convergence should feel free to simply treat the transform as a linear operator on power series (at least as far as this paper goes, very little would be lost in doing so).
In fact, many of the technical details in Section~\ref{sec:convolutions} (the 
computations regarding the domain of integration where the boundary ends up 
disappearing from the solution) are for the most part showing that we can 
safely treat the Laplace transform in exactly this way.

\subsection{Legendre Transform}\label{sec:Legendre}

Let $f$ be a function that is convex on an interval $X \subseteq \R$.
The {\em Legendre transform} is defined to be the function 
\begin{equation}\label{eqn:legendre}
f^*(s) 
= \sup_{x \in X} \left\{  xs - f(x) \right\}
\end{equation}
where the domain of $f^*$ is the space
\[
X^*
= \left \{x^* \in \R : \sup_{x \in X} \{ x x^*- f(x) \} <\infty \right \}.
\]
Note that the convention is to have $f^*$ be a function in the variable $p$, but we will use $p$ for other purposes, and so we will use the variable $s$.
In the case that $f$ is differentiable, one has the following relation:
\begin{lemma}\label{lem:legendre}
Let $f$ be strictly convex on $X$ and differentiable at a point $z \in X$.
Then $f'(z) \in X^*$ and
\[
f^*(f'(z)) 
= zf'(z) - f(z).
\]
\end{lemma}
\begin{proof}
Let $z \in X$.
Since $f$ is differentiable at $z$ and strictly convex, it satisfies the inequality 
\[
f(x) 
\geq f(z) + (x - z)f'(z)
\]
for all $x \in X$, with equality if and only if $x = z$.
Rearranging, we have
\[
\sup_{x \in X} \{ xf'(z) - f(x) \} 
= zf'(z) - f(z) < \infty
\]
which, by definition, means $f'(z) \in X^*$ and $f^*(f'(z)) = zf'(z) - f(z)$.
\end{proof}

Many of the useful properties of the Legendre transform follow directly from Lemma~\ref{lem:legendre}.
In particular, it implies that $f^*$ is strictly convex whenever $f$ is twice differentiable.

\begin{corollary}\label{cor:legendre2}
Let $f$ satisfy the conditions of Lemma~\ref{lem:legendre}.
Then 
\begin{enumerate}
\item $f'({f^*}'(x)) = x$
\item $f''({f^*}'(x)) = 1 / {f^*}''(x)$
\end{enumerate}
\end{corollary}
\begin{proof}
Take derivatives of the formula in Lemma~\ref{lem:legendre}.
\end{proof}

The Legendre transform appears in many areas under many different names.
It is often referred to as the {\em convex conjugate} in the analysis literature and as the {\em Fenchel transform} (or {\em Legendre--Fenchel transform}) in optimization.
In particular, the term ``convex conjugate'' is typically applied when one applies Equation~(\ref{eqn:legendre}) to more general Banach spaces, and the term ``Fenchel dual'' is typically applied when one applies Equation~(\ref{eqn:legendre}) in primal--dual algorithms of convex programming. 
In this paper, we use only the most basic facts from theory, and as result, we will maintain the nomenclature {\em Legendre transform} to highlight this fact. 

\subsection{\texorpdfstring{$L^p$}{Lp} norms}\label{sec:norms}

The main tool we will use is the theory of $L^p$ spaces from Banach space theory.
Given a measure space $(X, \mu)$, 
For $0 < p < \infty$, the $L^p$-norm of a function $f$ is defined to be
\[
\mynorm{f}{p} 
= \left( \int |f|^p \d{\mu} \right)^{\frac{1}{p}}
\]
and for $p = \infty$, we have 
\[
\infnorm{f} 
= \lim_{p \to \infty} \mynorm{f}{p} 
= \inf \left\{ a \geq 0 : \mu(\{ x : |f(x)| > a \}) = 0 \right\}
\]
We will only use one simple result from the theory of $L^p$ spaces:
\begin{lemma}\label{lem:inf_norm}
If $\mu$ is absolutely continuous with respect to the Lebesgue measure and $f$ is continuous, then
\[
\infnorm{f} 
= \sup_{x \in X} \{ |f(x)| \}
\]
\end{lemma}

The following simple observation relates the Legendre transform to $L^{p}$ spaces.
\begin{corollary}\label{cor:sup_norm}
Let $X$ be a subset of the real line and $\mu$ a measure that is absolutely continuous with respect to the Lebesgue measure.
Then for any continuous function $f: X \to \R$, we have
\[
f^*(s) 
= \ln \infnorm{ e^{xs - f(x)} }
\]
for all $s \in X^*$ (where both the Legendre transform and $L^{\infty}$ norm are taken over the space $X$).
\end{corollary}
\begin{proof}
Since $f$ is real, $f^*$ is real as well.
Hence we can write $f^*(s)$ as 
\begin{align*}
f^*(s) 
&= \ln(\exp(f^*(s))
\\&= \ln \left(\exp \left( \sup_{x \in X}\left\{  xs - f(x) \right\} \right) \right)
\\&= \ln \left( \sup_{x \in X}\left\{  \exp( xs - f(x)) \right\} \right).
\end{align*}
Since $f$ is real-valued, $\exp( xs - f(x))$ is strictly nonnegative, and so 
\[
\ln \left( \sup_{x \in X}\left\{  |\exp( xs - f(x))| \right\} \right) 
= \ln \infnorm{ e^{xs - f(x)} }
\]
as required.
\end{proof}

For what follows, we will always have $\mu$ be the Lebesgue measure with the 
domain of integration assumed to be the real line.
In the case that we want to integrate on a smaller domain, we will explicitly set the function to $0$ as in Equation~(\ref{eq:fkdet}).
We will also use the convention that all norms should be taken with respect to the variable $x$.
That is, when we write
\[
\mynorm{f(x, s)}{p}
\]
we mean 
\[
\mynorm{f(\cdot, s)}{p}
\]
and so the former should always be seen as a function of $s$.

\subsection{Fuglede--Kadison determinants}\label{sec:FK}

Rather than give the general definition of the Fuglede--Kadison determinant in 
operator theoretic terms, we will simply define it for the case in which we 
will need it.
For a finite dimensional $\M \times \M$ positive definite matrix $A$, we define the {\em normalized determinant} to be
\begin{equation}\label{eq:fkdet}
\fkdet{A} = 
\begin{cases}
\mydet{A}^{1/\M} & \text{if $A$ is positive definite} \\
 0 & \text{otherwise}
\end{cases}
\end{equation}
Fuglede and Kadison showed that for a sequence of positive definite matrices 
$A_1, A_2, \dots, $ for which the spectral measure converges in a suitable way 
to a limiting operator $a$, the normalized determinant converges to a 
well-defined limit and the limiting operator has many of the usual properties 
that one expects in a determinant (for example, multiplicativity) 
\cite{FugledeKadison}.

The Fuglede--Kadison determinant is typically defined in much greater generality, however we will only need the case where it acts as a normalized version of the determinant function (the goal being to trade generality (and perhaps rigor) for readability by those not familiar with $C^*$-algebras).
The only nontrivial property that we will use is that $\fkdet{}$ is well-defined and is a limit of its values on a suitable sequence of matrices.
It should be noted that the Fuglede--Kadison determinant (in the greater generality) already plays a role in the theory of free probability in the form of {\em Brown measures} (see \cite{deninger} for a survey of such connections).
It is unclear whether there is any relationship between the methods in this paper and such results.

\subsection{Mixed discriminants}\label{sec:md}

Let $X_1, \dots, X_{\M}$ be $\M \times \M$ matrices.
We define the {\em mixed discriminant} to be
\[
\D{X_1, \dots, X_{\M}} 
= \frac{\partial^{\M}}{\partial t_1 \dots \partial t_{\M}} \mydet{ \sum_{i=1}^{\M} t_i X_i }
\]

\begin{remark}
Note that our definition of the mixed discriminant differs by a factor of $\M!$ from many other treatments.
The literature is far from standard in this respect, and our reason for taking this normalization is that it will simplify a number of the formulas we will use.
\end{remark}

The properties of mixed discriminants are well studied (see, for example \cite{bapat1989mixed}), and appear in numerous contexts.
Some of the more well-known properties are the following:
\begin{lemma}
Let $A, B, X_2, \dots, X_{\M}$ be $\M \times \M$ matrices and let $c$ be a scalar.
Then 
\begin{enumerate}
\item $\D{}$ is invariant under permutation of its arguments
\item $\D{A + cB, X_2, \dots, X_{\M}} = \D{A, X_2, \dots, X_{\M}} + c\D{B, X_2, \dots, X_{\M}}$
\item $\D{A, \dots, A} = \M! \mydet{A}$
\item $\D{AX_1, \dots, AX_{\M}} = \mydet{A} \D{X_1, \dots, X_{\M}}$
\end{enumerate}
\end{lemma}
One corollary of these properties is that the determinant can be decomposed in a manner similar to the binomial theorem: 
\begin{equation}\label{eq:binomial}
\mydet{xI + A} 
= \sum_i x^{\M-i} \binom{\M}{i} \D{A[i], I[\M-i]}.
\end{equation}

While the mixed discriminant clearly has a close relationship to the determinant, it is also closely related to the {\em permanent}.
In particular, when $X_1, \dots, X_{\M}$ are {\em diagonal} matrices, one gets the formula  
\begin{equation}\label{eq:permanent}
\D{X_1, \dots, X_{\M}} 
= \mathrm{Perm}(Q)
\end{equation}
where the columns of $Q$ are the diagonals of the $X_i$.

The only nonstandard property of the mixed discriminant that we will use was proven in \cite{bcms} and can be seen as a type of distributivity:
\begin{lemma}\label{lem:trace_formula}
Let $A$ and $X_1, \dots, X_{\M}$ be $\M \times \M$ matrices.
Then 
\[
\Tr{A}\D{X_1, \dots, X_{\M}} 
= \sum_{i=1}^{\M} \D{X_1, \dots, X_{i-1}, A X_i, X_{i+1}, \dots, X_{\M}}
\]
\end{lemma}

While we will be using the mixed discriminant in the context of real symmetric matrices, it is worth mentioning that all of the properties mentioned here hold for arbitrary square matrices (except Equation~(\ref{eq:permanent}), obviously, which requires the matrices to be diagonal).

\section{\marcus~Transform}\label{sec:marcus_transform}

Given a multiset $S$, we will write $|S|$ for the number of elements in the multiset (with multiplicity).
To be pedantic, $|\{ 1, 2, \iu, \iu \}| = 4$.
To abuse notation, we will often treat multisets as though they were random variables.
In any such case, the random variable should be considered to be uniformly distributed on the elements of the multiset.
For example, 
\[
\frac{1}{|S|} \sum_{s_i \in S} f(s_i)
\AND
\E{}{f(S)}
\]
will be used interchangably.

\begin{lemma}\label{lem:marcus_transform}
Let $S$ be a finite multiset of complex numbers with $|S| = \M$.
Then there exists a unique multiset $T$ of complex numbers with $|T| = \M$ such that
\[
\prod_{s_i \in S} (x - s_i) 
= \E{}{(x-T)^{\M}}.
\]
for all $x$.
\end{lemma}
\begin{proof}
Clearly both sides of the equation are monic polynomials, so it suffices to prove that the equality holds for each of the other $\M$ coefficients.
Each such equality can be seen as a constraint on $\E{}{T^k}$ for $1 \leq k \leq \M$.
Using Newton's identities, this is equivalent to having constraints on the first $\M$ elementary symmetric functions of the elements of $T$.
However, this is equivalent to having $\M$ solutions to a polynomial of degree $\M$, which is true (and unique) over the complex numbers.
\end{proof}
Given a multiset $S$, we will refer to the multiset $T$ which satisfies the constraints of Lemma~\ref{lem:marcus_transform} as the {\em \marcus~transform} of $S$.
The property of having $|T| = |S|$ is an important one, as the next lemma will show.
\begin{lemma}\label{lem:real}
Let $S$ be a multiset of real numbers and $T$ its \marcus~transform.
Then 
\[
\E{}{f(T)} \in \R.
\]
for any function $f$ that is analytic on the support of $T$.
\end{lemma}
\begin{proof}
Since all elements of $S$ are real, the coefficients of 
\[
\prod_{s_i \in S}^{\M} (x - s_i) 
= \sum_{k = 0}^{\M} \binom{\M}{k} x^{\M-k} (-1)^k \E{}{T^k}
\]
are all real.
Now consider the polynomial 
\[
q(x) 
= \prod_{t_i \in T}^{\M} (x - t_i)
= \sum_{i=0}^\M q_i x^{\M-i}.
\]
By the Newton identities, the coefficients $q_i$ are expressable as functions of the first $\M$ moments (which we have just seen are real) and are therefore real.
Now let $A_T$ be a real symmetric matrix with eigenvalues the elements of $T$.
Then the Cayley--Hamilton theorem asserts that (as a matrix equation) $q(A_T) = 0$.
Hence
\[
0 
= \tr{A_T q(A_T)} 
= \sum_i q_i \tr{A_T^{\M-i+1}}
\]
which is an expression for $\tr{A_T^{\M+1}}$ as a linear combination of $\tr{A_T^{k}}$ for $k \leq \M$ with real coefficients (and so itself is real).
By proceeding inductively, we have $\tr{A_T^k} = \E{}{T^k} \in \R$ for all $k$ and so the same will be true for any analytic function.
\end{proof}
Note that if $|T|$ were larger than $|S|$, there could be no such guarantees.
To see this, let $k > 0$.  
We can use the same logic as Lemma~\ref{lem:marcus_transform} to find a multiset $W$ with $|W| = (\M+k)$ which has the same first $\M$ moments as $\mu$, but then has a complex $(\M+1)^{th}$ moment.

Despite our willingness to treat multisets as distributions, one should be careful when going in the reverse direction.
We will say that a distribution $\mu$ is {\em $\M$-realizable} if there exists a multiset $S_\mu$ with $|S_\mu| = \M$ such that the uniform distribution on $S_{\mu}$ gives the same probabilities as $\mu$.
The multiset $S_\mu$ will then be referred to as its {\em $\M$-realization}.
In particular, the \marcus~transform should always be defined in terms of a realization of a distribution and not the distribution itself.
It is easy to see that an $\M$-realizable distribution will also be $k\M$-realizable for any positive integer $k$.
The realizations, however, will be (except in very special cases) quite different, as the next example shows.
\begin{example}
Let $\mu$ be the two-point distribution taking values $\{ -1, 1 \}$ each with probability $1/2$.
Then $S_2 = \{ -1, 1 \}$ is its $2$-realization, which has \marcus~transform $T_2 =\{ -\iu, \iu \}$ which can be seen by computing
\[
\E{}{(x - S_2)^2} = \frac{1}{2} (x^2 -2x\iu - 1) + \frac{1}{2} (x^2 + 2x\iu - 1) = x^2 -1 = (x + 1)(x - 1)
\]
However, the $4$-realization of $\mu$ ($S_4 = \{ 1, 1, -1, -1 \}$) has \marcus~transform
\[
T_4 = \left \{
\sqrt{\frac{2\sqrt{2}-1}{3}},
\sqrt{\frac{2\sqrt{2}-1}{3}},
\iu\sqrt{\frac{2\sqrt{2}+1}{3}},
-\iu\sqrt{\frac{2\sqrt{2}+1}{3}}
\right\}
\]
which can be checked by computing
\[
\E{}{T_4^4} 
= 1
\AND
\E{}{T_4^2} 
= -\frac{1}{3}
\]
so that 
\[
\E{}{(x - T_4)^4} 
= x^4 - 2x^2 + 1 
= (x - 1)^2 (x+1)^2.
\]
\end{example}

The utility of the \marcus~transform will lie in its ability to turn polynomial convolutions (and therefore, as we shall see, finite free independence) into classical independence.
This is illustrated in the following lemma:
\begin{lemma}
\label{lem:poly_conv}
Let $p$ and $q$ be degree $\M$ polynomials with $S$ and $T$ the \marcus~transforms of their roots.
If $S$ and $T$ are independent, then
\[
[ p \sqsum_{\M} q ](x) 
= \E{}{(x - S - T)^{\M}}
\AND
[ p \sqmult_{\M} q ](x) 
= \E{}{(x - ST)^{\M}}
\]
\end{lemma}
\begin{proof}
Let $p(x) = \sum_i x^{d-i} (-1)^i p_i$.
As observed previously, we have 
\[
p(x) = \E{}{(x - S)^{\M}} = \sum_i x^{\M-i} \binom{\M}{i} \E{}{(-S)^i}
\]
and so $p_i = \binom{\M}{i}  \E{}{S^i}$ (and similarly for $q$ and $T$).
By Equation~(\ref{eq:add_conv}), we have
\begin{align*}
[ p \sqsum_{\M} q ](x)
&=\sum_{i=0}^{\M} \sum_{j=0}^{\M-i} x^{\M-i-j} (-1)^{i + j} \frac{(\M-i)!(\M-j)!}{\M!(\M-i-j)!}p_i q_j
\\&= \sum_{i=0}^{\M} \sum_{j=0}^{\M-i} x^{\M-i-j} (-1)^{i + j} \frac{(\M-i)!(\M-j)!}{\M!(\M-i-j)!} \binom{\M}{i}\binom{\M}{j} \E{}{S^i}\E{}{T^j}
\\&= \sum_{i=0}^{\M} \sum_{j=0}^{\M-i} x^{\M-i-j} (-1)^{i + j} \binom{\M}{i, j} \E{}{S^iT^j} && \text{(independence)}
\\&= \E{}{(x - S - T)^{\M}}.
\intertext{The multiplicative case is similar, using Equation~(\ref{eq:mult_conv}):}
[ p \sqmult_{\M} q ](x)
&=\sum_{i=0}^{\M} x^{\M-i}(-1)^i \frac{p_iq_i}{\binom{\M}{i}}
\\&= \sum_{i=0}^{\M} x^{\M-i} (-1)^i \binom{\M}{i}\binom{\M}{i}\frac{\E{}{S^i}\E{}{T^i}}{\binom{\M}{i}} 
\\&= \sum_{i=0}^{\M} x^{\M-i} (-1)^i \binom{\M}{i}\E{}{S^iT^i} && \text{(independence)}
\\&= \E{}{(x - ST)^{\M}}.
\end{align*}
as required.
\end{proof}
An easy corollary of Lemma~\ref{lem:poly_conv} is scale and translation invariance of the \marcus~transform.
\begin{corollary}
Let $T$ be the \marcus~transform of a multiset $S$.
Then the \marcus~transforms of 
\[
\{ s + k : s \in S \}
\AND
\{ ks : s \in S \}
\]
are
\[
\{ t + k : t \in T \}
\AND
\{ kt : t \in T \}
\]
respectively.
\end{corollary}
\begin{proof}
Use Lemma~\ref{lem:marcus_transform} with $q = (x - k)^{\M}$.
\end{proof}

\section{Free probability as the limit of polynomial convolutions}\label{sec:convolutions}

\renewcommand{\deriv}[1]{ \frac{\partial}{\partial #1}}

The goal of this section is to show the relationship between the polynomials convolutions defined in Section~\ref{sec:formulas} and free probability.
Our approach will be to introduce sequences of transforms (indexed by positive integers) which converge to Voiculescu's R-transform and S-transform when applied to the spectral distributions of Hermitian operators.
This will define a sequence of convolutions which will converge to the free additive and free multiplicative convolution from free probability, but in the finite case will reduce to the convolutions of polynomials.

Before getting into any details, it is worth considering what to expect out of finite versions of (any) transforms.
Recall that the spectral distribution of an $\M \times \M$ matrix is completely determined by the first $\M$ moments.
As a result, the same is true for any function of these moments, including the transforms that we will define.
It will nonetheless be computationally beneficial to be able to consider generic power series, which in our context, will simply have extraneous information (that we can choose to ignore if and when it is useful to us).

Interestingly, we will find that the additive and multiplicative convolutions 
will ``store'' their moments in the coefficients of different polynomial bases.
The additive case will use the standard polynomial basis and so we will be 
interested in certain coefficients of our power series.
The multiplicative convolution, on the other hand, will use ``rising 
factorial'' bases, which in turn will cause our interest to lie in the 
evaluations of our power series at certain points (see Remark~\ref{rem:bases}).
The additive case, in particular, will require some operations that do not 
obviously preserve such information (as this is not true for general 
operations).
Hence we will need to prove that the operations that we will encounter can be 
performed without destroying information, which we do now.

Let $f(x) = \sum_i a_i x^i$ be a formal power series.
For an integers $k$, we will write $f \mymod{x}{k}$ to denote the polynomial 
\[
\sum_{i = 0}^{k-1} a_i x^i
\]
and we will write 
\[
f \equiv g \mymod{x}{k}
\]
if $f - g \mymod{x}{k}$ is the $0$ polynomial.
Many properties can be derived easily from the definition.
For example, if
\[
a(x) \equiv b(x) \mymod{x}{k} 
\AND
c(x) \equiv d(x) \mymod{x}{k}
\]
then we have 
\[
a(x)c(x) \equiv b(x)d(x) \mymod{x}{k}
\AND
a(x) + c(x) \equiv b(x) + d(x) \mymod{x}{k}
\]
Combining these implies (under the same assumptions), that
\[
h(a(x)) \equiv h(b(x)) \mymod{x}{k}
\]
for any power series $h$.
This leads to the following observation, which we state as a corollary:
\begin{corollary}
\label{cor:mod_comp}
Let $a, b, h$ be power series with $h$ invertible (that is, there exists power series $g$ such that $h(g(x)) = x$ for all $x$.
Then 
\[
a(x) \equiv b(x) \mymod{x}{k} \iff h(a(x)) \equiv h(b(x)) \mymod{x}{k}  
\]
\end{corollary}

\subsection{Additive convolution}

We begin with the additive case.
\begin{definition}
Let $A$ be a Hermitian operator with with compactly supported spectral distribution $\mu_A$.
For an integer $\M$, we define the power series
\begin{equation}
\label{eq:myK}
\myinvcauchy{\M}{\mu_A}{s} = -\deriv{s} \ln \mynorm{ e^{-xs} \fkdet{xI - A}}{\M}
\end{equation}
where the domain of integration is $(\rho_A, \infty)$.
We call $\myinvcauchy{\M}{\mu_A}{s}$ the {\em $\M$-finite K-transform of $\mu_A$}.
\end{definition}
It should be clear from the definition that $\myinvcauchy{\M}{\mu_A}{s}$ is 
invariant under unitary transformations of $A$.
The $\M$-finite K-transform will be the analogue of the inverse Cauchy transform from Voiculescu's theory.
We then define the {\em $\M$-finite R-transform} by
\begin{equation}
\label{eq:myR}
\myrtrans{\M}{\mu_A}{s} = \myinvcauchy{\M}{\mu_A}{s} - \myinvcauchy{\M}{\mu_0}{s}
\end{equation}
where $\mu_0$ is the constant $0$ distribution.
It is not hard to calculate (and we will do most of it in Lemma~\ref{lem:add_compute}) that 
\[
\myinvcauchy{\M}{\mu_0}{s} = \left( 1 + \frac{1}{\M} \right)\frac{1}{s}
\]
which one can view as the discrete version of the familiar $1/s$ term that is subtracted from the inverse Cauchy transform to get the R-transform in Voiculescu's theory. 

\subsubsection{Relation to polynomial convolutions}\label{sec:add_conv}

We begin by showing the connection between finite R-transforms and finite free additive convolutions of polynomials.
Our first job will be to find the value of the $\M$-finite R-transform of a distribution on an $\M \times \M$ matrix.
The computation will employ the Laplace transform from Section~\ref{sec:Laplace} as well as the \marcus~transform introduced in Section~\ref{sec:marcus_transform}.

\begin{lemma}
\label{lem:add_compute}
If $A$ is an $\M \times \M$ Hermitian matrix, then 
\[
\frac{\mynorm{ e^{-xs} \fkdet{xI - A} }{\M}^{\M}}{\mynorm{ e^{-xs} \fkdet{xI - 0} }{\M}^{\M}}
\equiv
\E{}{e^{-{\M}sT_A}} \mymod{s}{\M+1}
\]
where $T_A$ is the \marcus~transform of $\eigen{A}$.
\end{lemma}
\begin{proof}
For an $\M \times \M$ matrix $A$, we have the simplification
\[
\fkdet{xI - A}^{\M} = \mydet{xI - A} = \E{}{(x - T_A)^{\M}}.
\]
where $T_A$ is the \marcus~transform of $\eigen{A}$.
Hence
\begin{align*}
\mynorm{ e^{-xs} \fkdet{xI - A} }{\M}^{\M}
&= \int_{\rho_A}^{\infty} e^{-{\M}xs} \E{}{(x - T_A)^{\M}} \d{x} && (y := x - \rho_A)
\\&= e^{-{\M}\rho_A s}\int_{0}^{\infty} e^{-{\M}ys} \E{}{(y+\rho_A - T_A)^{\M}} \d{x}
\\&= \E{}{e^{-{\M}\rho_A s}\L{(x + \rho_A - T_A)^{\M}}({\M}s)}
\end{align*}
where
\begin{align*}
\L{(x + \rho_A - T_A)^{\M}}({\M}s)
&= \sum_{i=0}^{\M} \binom{\M}{i}(\rho_A-T_A)^{\M-i} \L{x^i}({\M}s)
\\&= \sum_{i=0}^{\M} \binom{\M}{i}(\rho_A-T_A)^{\M-i} \frac{i!}{({\M}s)^{i+1}}
\\&= \M! \sum_{i=0}^{\M} (\rho_A-T_A)^{\M-i} \frac{({\M}s)^{-i-1}}{(\M-i)!}.
\end{align*}
When $A$ is the $0$ matrix, we have $\rho_A = T_A = 0$, so 
\[
\mynorm{ e^{-xs} \fkdet{xI - A} }{\M}^{\M} = \frac{\M!}{(\M s)^{\M+1}}.
\]
Hence we can write
\begin{align*}
\frac{\mynorm{ e^{-xs} \fkdet{xI - A} }{\M}^{\M}}{\mynorm{ e^{-xs} \fkdet{xI - 0} }{\M}^{\M}}
&= \frac{(\M s)^{\M+1}}{\M!} \E{}{e^{-{\M}s \rho_A}\L{(x + \rho_A - 
T_A)^{\M}}({\M}s)}
\\&= e^{-{\M}s \rho_A} \E{}{\sum_{i=0}^{\M} (\rho_A-T_A)^{\M-i} 
\frac{({\M}s)^{\M-i}}{(\M-i)!}}.
\end{align*}
where
\[
\sum_{i=0}^{\M} (\rho_A-T_A)^{\M-i} \frac{({\M}s)^{\M-i}}{(\M-i)!} \equiv \E{}{e^{(\rho_A - T_A)\M s}} \mymod{s}{\M+1}
\]
Hence
\begin{align*}
\frac{\mynorm{ e^{-xs} \fkdet{xI - A} }{\M}^{\M}}{\mynorm{ e^{-xs} \fkdet{xI - 0} }{\M}^{\M}}
&\equiv e^{-{\M}\rho_A s} \E{}{e^{(\rho_A - T_A)\M s}} \mymod{s}{\M+1}
\\&\equiv \E{}{e^{-\M s T_A}} \mymod{s}{\M+1} 
\end{align*}
as claimed.
\end{proof}
This gives us a direct formula for the $\M$-finite R-transform:
\begin{corollary}
\label{cor:rtrans}
If $A$ is a $\M \times \M$ Hermitian matrix, then 
\[
\myrtrans{\M}{\mu_A}{s} \equiv -\frac{1}{\M}\deriv{s} \ln \E{}{e^{-{\M}sT_A}} 
\mymod{s}{\M}
\]
where $T_A$ is the \marcus~transform of $\eigen{A}$.
\end{corollary}
\begin{proof}
We start by unpacking Equation~(\ref{eq:myR}):
\begin{align*}
\myrtrans{\M}{\mu_A}{s} 
&= \myinvcauchy{\M}{\mu_A}{s} - \myinvcauchy{\M}{\mu_0}{s}
\\&= -\deriv{s} \ln \mynorm{ e^{-xs} \fkdet{xI - A}}{\M} + \deriv{s} \ln \mynorm{ e^{-xs} \fkdet{xI - 0}}{\M}
\\&= -\frac{1}{\M}\deriv{s} \ln \left( \frac{\mynorm{ e^{-xs} \fkdet{xI - A}}{\M}^{\M}}{\mynorm{ e^{-xs} \fkdet{xI - 0} }{\M}^{\M}} \right).
\end{align*}
By Lemma~\ref{lem:add_compute}, we have 
\[
\frac{\mynorm{ e^{-xs} \fkdet{xI - A} }{\M}^{\M}}{\mynorm{ e^{-xs} \fkdet{xI - 0} }{\M}^{\M}}
\equiv
\E{}{e^{-{\M}sT_A}} \mymod{s}{\M+1}
\]
where $T_A$ is the \marcus~transform of $\eigen{A}$.
Since $\ln$ is invertible, we can apply Corollary~\ref{cor:mod_comp} to get
\[
\ln \left(\frac{({\M}s)^{\M+1}}{\M!} \mynorm{ e^{-xs} \fkdet{xI - A}}{\M}^{\M} \right) \equiv \ln \E{}{(s - T_A)^{\M}} \mymod{s}{\M+1}.
\]
This implies the claim, since having two power series match on the first $\M+1$ coefficients implies their derivatives match on the first $\M$ coefficients. 
\end{proof}

\begin{lemma}
\label{lem:sum_reduction}
Let $A$ and $B$ be $\M \times \M$ Hermitian matrices.
Then the following are equivalent:
\begin{enumerate}
\item $\myrtrans{\M}{\mu_A}{s} + \myrtrans{\M}{\mu_B}{s} \equiv \myrtrans{\M}{\mu_{A + B}}{s} \mymod{s}{\M}$
\item $\mydet{xI - A} \sqsum_{\M} \mydet{xI - B} = \mydet{xI - A - B}$
\end{enumerate}
\end{lemma}
\begin{proof}
Let $T_A, T_B, T_{A+B}$ be the \marcus~transforms of $\eigen{A}, \eigen{B}, \eigen{A+B}$ respectively and where $T_A$ and $T_B$ are treated as independent random variables.
By Lemma~\ref{lem:poly_conv}, we have that {\em 2.} holds if and only if
\[
\E{}{(x - T_A - T_B)^{\M}}  = \E{}{(x-T_{A+B})^{\M}}
\]
which holds if and only if $T_A + T_B$ and $T_{A + B}$ have the same first $\M$ moments.
This in turn is equivalent to the statement
\[
\E{}{e^{-{\M}s(T_A + T_B)}} \equiv \E{}{e^{-{\M}sT_{A+B}}} \mymod{s}{\M+1}
\]
which is true if and only if 
\[
\E{}{e^{-{\M}s(T_A)}} \E{}{e^{-{\M}s(T_B)}} \equiv \E{}{e^{-{\M}sT_{A+B}}} \mymod{s}{\M+1}
\]
since $T_A$ and $T_B$ were chosen to be independent.
Since $\ln$ is invertible, we can apply to Corollary~\ref{cor:mod_comp} to see that this is equivalent to 
the statement
\[
f_A(s) + f_B(s) \equiv f_{A+B}(s) \mymod{s}{\M+1}
\]
where $f_X(s) = -\frac{1}{\M} \ln \E{}{e^{-{\M}s(T_B)}}$.
So by Corollary~\ref{cor:rtrans}, it remains to show that the two statements
\begin{enumerate}
\item $f_A(s) + f_B(s) = f_{A+B}(s) \mymod{s}{\M+1}$
\item $\deriv{s} f_A(s) + \deriv{s} f_B(s) \equiv \deriv{s} f_{A+B}(s) \mymod{s}{\M}$
\end{enumerate}
are equivalent.
The forward implication is obvious (and was done in the proof of Corollary~\ref{cor:rtrans}).
The only issue with the reverse implication, however, is the constant term.
Hence we merely need to show that $f_A(0) + f_B(0) = f_{A+B}(0)$, and we would be done.
However this follows easily by noting that (by definition) $f_X(0) = 0$ for all $X$.
\end{proof}

Note that Lemma~\ref{lem:sum_reduction} relied heavily on the fact that $A$ and $B$ were $\M \times \M$ matrices.
This came in the form of the assertion that the first $\M$ moments completely characterize the distribution (which of course is not true for more general distributions).

\subsubsection{Relation to Voiculescu}

We wish to show that the our definition is, in fact, a generalization of Voiculescu's R-transform.
We first show that this is the case for a fixed distribution.
\begin{lemma}
\label{lem:converge_to_k}
Let $A$ be a Hermitian operator with compactly supported spectral distribution $\mu_A$.
Then
\[
\lim_{\M \to \infty} \myinvcauchy{\M}{\mu_A}{s} = \invcauchy{\mu_A}{s}
\]
at all points $s \in (\rho_A, \infty)$.
\end{lemma}
\begin{proof}
We first note that the function
\[
f(x) = -\ln \fkdet{xI - A}
\]
exists and is continuous on $(\rho_A, \infty)$.
Furthermore, we have
\[
f''(x) = \tr{ (xI - A)^2 } > 0
\]
and so $f$ is strictly convex.
By Corollary~\ref{cor:legendre2}, we therefore have ${f'}^{-1}(s) = {f^*}'(s)$
where
\[
f'(x) = -\tr{(xI - A)^{-1}} = -\cauchy{\mu_A}{x}
\]
and $\cauchy{\mu_A}{x}$ is the Cauchy transform.
Hence
\[
{f'}^{-1}(s) = \left( -\cauchy{\mu_A}{x} \right)^{-1}(s) = \invcauchy{\mu_A}{-s}.
\]
Plugging this into Corollary~\ref{cor:sup_norm} gives
\[
\invcauchy{\mu_A}{s} = {f^*}'(-s) = -\deriv{s} \ln \infnorm{ e^{-xs} \fkdet{xI - A}}
\]
The result then follows from Lemma~\ref{lem:inf_norm}.
\end{proof}

\subsection{Multiplicative convolution}\label{sec:mult}

We follow the same path as in the additive case.
Note that this time we require the spectral distribution to be positive almost surely, just as in Voiculescu's theory.

\begin{definition}
Let $A$ be a positive definite operator with compactly supported spectral distribution $\mu_A$.
For an integer $\M$, we define the power series $\myinvmtrans{\M}{\mu_A}{s}$ by
\begin{equation}
\label{eq:myN}
\ln \myinvmtrans{\M}{\mu_A}{s} = -\deriv{s} \ln \mynorm{ e^{-xs} \fkdet{I - e^{-x}A} }{\M}
\end{equation}
with the domain of integration being $(\ln \rho_A, \infty)$.
We call $\myinvmtrans{\M}{\mu_A}{s}$ the {\em $\M$-finite N-transform of 
$\mu_A$}.
\end{definition}

Again it should be clear that $\myinvmtrans{\M}{\mu_A}{s}$ is invariant under unitary transformations of $A$.
Similar to before, the $\M$-finite N-transform will be the analogue of the inverse M-transform from Voiculescu's theory.
We then define the {\em $\M$-finite S-transform} by
\begin{equation}
\label{eq:myS}
\ln \mystrans{\M}{\mu_A}{s} = \ln \myinvmtrans{\M}{\mu_A}{s} - \ln \myinvmtrans{\M}{\mu_I}{s}
\end{equation}
where $I$ is the identity operator.
Note that we have dropped the ``modified'' adjective since no other such thing exists in this context.
This time calculating $\myinvmtrans{\M}{\mu_I}{s}$ is a bit more involved; using the calculation in Lemma~\ref{lem:mult_reduction}, we get
\[
\ln \myinvmtrans{\M}{\mu_I}{s} = -\frac{1}{\M} \deriv{s} \ln \left(\frac{\Gamma({\M s})\M!}{\Gamma(\M s + \M + 1)}\right) = \psi(\M s + \M + 1) - \psi(\M s)
\]
where $\psi(x) = \deriv{s} \ln \Gamma(x)$ is the {\em digamma function}.
Once again one can view this as the discrete version of the familiar $(s+1)/s$ term that is multiplied with the inverse M-transform to get the S-transform in Voiculescu's theory.
To see this, we can use the standard asymptotic approximation for the digamma function $\psi(x) = \ln(x) + O(1/x)$, giving
\[
\lim_{\M \to \infty} \myinvmtrans{\M}{\mu_I}{s} = \lim_{\M \to \infty} \frac{\M s + \M + 1}{\M s} = \frac{s+1}{s}
\]
as expected. 

\subsubsection{Relation to polynomial convolutions}

As before, we start by computing the case that $A$ is an $\M \times \M$ matrix.
\begin{lemma}
\label{lem:mult_compute}
Let $A$ be an $\M \times \M$ positive definite matrix with $T_A$ the \marcus~transform of $\eigen{A}$.
Then 
\[
\frac{\mynorm{ e^{-xs} \fkdet{I - e^{-x}A} }{\M}^{\M}}{\mynorm{ e^{-xs} \fkdet{I - e^{-x}I} }{\M}^{\M}}
=
\rho_A^{-\M s}f_A(s)
\]
where $f_A(x)$ is the unique degree $\M$ polynomial which satisfies
\[
f_A\left(-\frac{k}{\M}\right) = \E{}{\left(\frac{T_A}{\rho_A}\right)^k}
\]
for all integers $0 \leq k \leq \M$.
\end{lemma}
\begin{proof}
We expand
\begin{align*}
\mynorm{ e^{-xs} \fkdet{I - e^{-x}A} }{\M}^{\M} 
&= \int_{\ln \rho_A}^{\infty} e^{-x{\M}s} \E{}{(1 - e^{-x}T_A)^{\M}} \d{x} && 
(y := x - \ln \rho_A)
\\&= \rho_A^{-{\M}s} \int_{0}^{\infty} e^{-y{\M}s} \E{}{(1 - e^{-y}T_A/\rho_A)^{\M}} \d{y}
\\&= \rho_A^{-{\M}s} \L{\E{}{(1 - e^{-y}T_A/\rho_A)^{\M}}}(\M s)
\end{align*}
where
\[
\L{(1 - e^{-x}T_A/\rho_A)^{\M}}({\M}s)
=
\sum_{i=0}^{\M} \binom{\M}{i}(-T_A/\rho_A)^{i}\L{e^{-ix}}({\M}s)
= 
\sum_{i=0}^{\M} \binom{\M}{i}\frac{(-T_A/\rho_A)^{i}}{{\M}s + i}.
\]
so that 
\begin{equation}
\label{eq:poles}
\mynorm{ e^{-xs} \fkdet{I - e^{-x}A} }{\M}^{\M}
=
 \rho_A^{-{\M}s} \sum_{i=0}^{\M} \binom{\M}{i}\frac{\E{}{(-T_A/\rho_A)^{i}}}{{\M}s + i}
\end{equation} 
Now when $A = I$ (the identity), we have $\ln \rho_A = 0$ and $T_A = 1$, so 
that 
\[
\mynorm{ e^{-xs} \fkdet{I - e^{-x}I} }{\M}^{\M} = \sum_{i=0}^{\M} \binom{\M}{i}\frac{(-1)^{i}}{{\M}s + i}.
\]
To find a closed form solution to this sum, we can set
\[
f(x) = \sum_{i=0}^{\M} \frac{(-1)^i \binom{\M}{i}x^{\M s + i}}{\M s + i}
\]
so that 
\[
f'(x) = \sum_{i=0}^{\M} (-1)^i \binom{\M}{i}x^{\M s + i - 1} = x^{\M s -1}(1-x)^{\M}.
\]
Hence 
\[
f(1) - f(0) = \int_{0}^1 x^{\M s-1}(1-x)^{\M} \d{x} = \beta(\M s, \M+1)
\]
where $\beta$ is the well known {\em Beta function}.
Assuming $s > 0$, we have $f(0) = 0$ and so 
\[
f(1) = \int_{0}^1 x^{\M s -1}(1-x)^{\M} \d{x} 
= \frac{\Gamma({\M s})\Gamma(\M + 1)}{\Gamma(\M s + \M + 1)} 
= \M!\frac{\Gamma({\M s})}{\Gamma(\M s + \M + 1)}.
\]
Hence we have 
\[
\frac{\mynorm{ e^{-xs} \fkdet{I - e^{-x}A} }{\M}^{\M}}{\mynorm{ e^{-xs} \fkdet{I - e^{-x}I} }{\M}^{\M}}
=
\frac{\rho_A^{-\M s}}{\M!}\frac{\Gamma(\M s + \M + 1)}{\Gamma({\M s})}\sum_{i=0}^{\M} \binom{\M}{i}\frac{\E{}{(-T_A/\rho_A)^{i}}
}{\M s + i}.
\]
Since 
\[
\frac{\Gamma(\M s + \M + 1)}{\Gamma({\M s})} = \prod_{i=0}^{\M} (\M s + i)
\]
the poles that appear in Equation~(\ref{eq:poles}) will be eliminated.
In particular, we can write 
\[
f_A(s) := \rho_A^{\M s}\frac{\mynorm{ e^{-xs} \fkdet{I - e^{-x}A} }{\M}^{\M}}{\mynorm{ e^{-xs} \fkdet{I - e^{-x}I} }{\M}^{\M}}.
\]
where $f_A(s)$ is a degree $\M$ polynomial. 
One can then check that
\[
f_A\left(-\frac{k}{\M}\right) = \frac{1}{\M!} \binom{\M}{k} 
\E{}{(-T_A/\rho_A)^{k}} \prod_{\substack{i=0 \\ i\neq k}}^d (i-k) = 
\E{}{\left(\frac{T_A}{\rho_A} \right)^{k}}
\]
for integers $0 \leq k \leq \M$, which therefore uniquely determines it.   
\end{proof}

Note that if one started with a collection of $\M + 1$ points
\[
\left( -\frac{k}{\M}, \E{}{\left(\frac{T_A}{\rho_A} \right)^{k}} \right),
\]
and used the method of Lagrange interpolation to build a degree $\M$ polynomial going through those points, the formula would produce exactly $f_A(s)$.
Of course it would have to produce $f_A(s)$ in some form (since it is uniquely 
determined), but in situations where there were not enough points to completely 
determine the polynomial, this observation might be useful.
%
%

\begin{remark}\label{rem:bases}

There is a noticeable difference between Lemma~\ref{lem:add_compute} (which 
produces a generating function characterized by its coefficients) and 
Lemma~\ref{lem:mult_compute} (which produces a polynomial characterized by 
evaluations).
On the other hand, the (Voiculescu) S-transform seemingly keeps all of its 
information in its coefficients, similar to the way the (Voiculescu) 
R-transform does, and so one can ask how this difference is resolved.

In this respect, we can turn to Taylor's theorem --- another way to think about 
coefficients in an expansion is as the collection of derivatives at $0$.
However Taylor's theorem can be extended to other polynomial bases using the 
theory of umbral calculus \cite{umbral}.
In particular, if one defines the {\em $t$-backward difference operator} as
\[
\nabla_t[f](x) = \frac{f(x) - f(x-t)}{t}
\]
then the umbral analogue to Taylor's theorem is
\[
f(x) 
= \sum_{k=0}^\infty \nabla^k_t[f](a) \frac{(x-a)(x-a-t)(x-a-2t) \dots 
(x-a-(k-1)t)}{k!}
:= \sum_{k=0}^\infty \nabla^k_t[f](a) \frac{p_{t,k}(x-a)}{k!}
\]
with the polynomials $p_{t, k}(x)$ forming a basis for $\R[x]$ (in the case 
that $t = 1$, this is called the ``rising factorial'' basis).
Hence the evaluations of $f_A$ in Lemma~\ref{lem:mult_compute} 
are in 1-1 correspondence with the values $\{ \nabla^k_{1/\M}[f_A](0) \}$.
Furthermore, it is easy to see that 
\[
\lim_{\M \to \infty} \nabla^k_{1/\M}[f](x) = \frac{\partial^k f}{(\partial 
x)^k}(x)
\AND
\lim_{\M \to \infty} p_{1/\M,k}(x)  = x^k
\]
and so the bases converge in the limit.
\end{remark}

\begin{lemma}\label{lem:mult_reduction}
Let $A$ and $B^{-1}$ be $\M \times \M$ positive definite matrices.
Then the following are equivalent:
\begin{enumerate}
\item $\mystrans{\M}{\mu_A}{-k/m}\mystrans{\M}{\mu_B}{-k/m} = \mystrans{\M}{\mu_{AB}}{-k/m}$ for all $0 \leq k \leq m$
\item $\mydet{xI - A} \sqmult_{\M} \mydet{xI - B} = \mydet{xI - AB}$
\end{enumerate}
\end{lemma}
\begin{proof}
Similar to before, let $T_A, T_B, T_{AB}$ be the \marcus~transforms of $\eigen{A}, \eigen{B}, \eigen{AB}$ respectively and treating $T_A$ and $T_B$ as independent random variables.
Plugging in the definition of the finite S-transform and integrating, Equation~{\em 1.} is equivalent to having
\[
\mynorm{ e^{-xs} \fkdet{I - e^{-x}A} }{\M}\mynorm{ e^{-xs} \fkdet{I - e^{-x}B} }{\M} = c\mynorm{ e^{-xs} \fkdet{I - e^{-x}AB} }{\M}\mynorm{ e^{-xs} \fkdet{I - e^{-x}I} }{\M}
\]
for some constant $c$.
By Lemma~\ref{lem:mult_compute} this can be rewritten as  
\begin{equation}
\label{eq:compare}
(\rho_A\rho_B)^{-\M s} f_A(s) f_B(s) = c(\rho_{AB})^{-\M s} f_{AB}(s) f_I(s)
\end{equation}
where the polynomials $f_X$ have the property that $f_X(-k/m) = \rho_X^{-k} \E{}{ T_X^k }$ (where $T_I = 1$).
Hence plugging in $-k/m$ into Equation~(\ref{eq:compare}) gives
\[
\E{}{T_A^k}\E{}{T_B^k} = c\E{}{T_{AB}}^k
\]
where the value of $c$ can be deduced from the $k=0$ case (so $c = 1$).
Since $T_A$ and $T_B$ are independent, the previous holds if and only if $T_AT_B$ and $T_{AB}$ have the same first $\M$ moments.
This is equivalent to the statement 
\[
\E{}{(x - T_AT_B)^{\M}} = \E{}{(x - T_{AB})^{\M}}.
\]
which is equivalent to {\em 2.} by Lemma~\ref{lem:poly_conv}.
\end{proof}
%

\subsubsection{Relation to Voiculescu}

We now show that the our definition is, again, a generalization of the modified S-transform (see Section~\ref{sec:free_prob} for the difference between this and Voiculescu's version).

\begin{lemma}
\label{lem:converge_to_n}
Let $A$ be a positive definite operator with compactly supported spectral distribution $\mu_A$.
Then
\[
\lim_{\M \to \infty} \myinvmtrans{\M}{\mu_A}{s} = \invmtrans{\mu_A}{s}
\]
at all points $s \in (\ln \rho_A, \infty)$.
\end{lemma}
\begin{proof}
This time we consider the function
\[
f(x) = -\ln \fkdet{I - e^{-x}A}.
\]
$f(x)$ exists and is continuous on the interval $(\ln \rho_A, \infty)$ and we have
\[
f''(x) = \tr{Ae^{-x}(I - Ae^{-x})^{-1}} + \tr{A^2e^{-2x}(I - Ae^{-x})^{-2}}
\]
which is strictly positive for $x > \ln \rho_A$ and so $f$ is strictly convex.
By Corollary~\ref{cor:legendre2}, we therefore have
\[
{f^*}'(s) = {f'}^{-1}(s) = \left( -\mtrans{\mu_A}{e^x} \right)^{-1}(s) = \ln \invmtrans{\mu_A}{-s}
\]
where 
\[
f'(x) = -\tr{Ae^{-x}(I - Ae^{-x})^{-1}} = \tr{I - e^x(e^xI - A)^{-1}} = 1 - e^x \cauchy{\mu_A}{e^x}  = -\mtrans{\mu_A}{e^x}.
\]
Plugging this into Corollary~\ref{cor:sup_norm} gives
\[
\ln \invmtrans{\mu_A}{s} = {f^*}'(-s) = -\deriv{s} \ln \infnorm{ e^{-xs} \fkdet{I - e^{-x}A} }
\]
and so the lemma then follows from Lemma~\ref{lem:inf_norm}.
\end{proof}

Note that when $A$ is a matrix with eigenvalues $r_1, \dots, r_{\M}$, the function $\fkdet{I - e^{-x}A}$ has zeroes exactly when $x = \ln r_i$.
Hence this seems to be encoding the fact that the addition of logs is the log of multiplication without actually operating on the log of the operator itself.

\section{Finite freeness}\label{sec:ff_position}

One of the appealing attributes of free probability is the fact that one can ``instantiate'' freeness.
That is, given two distributions $\mu_A$ and $\mu_B$, one can find operators $A$ and $B$ with those spectral distributions such that $\mu_A \boxplus \mu_B = \mu_{A+B}$ \cite{voiculescu}.
Of course finding such instances is hard (hence the need for convolutions).
That said, the mere knowledge that an additive convolution {\em could} be achieved by simple addition is quite useful, as there is much known about the relationship between the eigenvalues of a sum and the eigenvalues of the summands.
For example, one can deduce a trivial bound on the spectral radius of the sum given the addends $\rho_{A+B} \leq \rho_A + \rho_B$, an inequality that is not directly obvious from the definition of the additive convolution.

In this section we will attempt to define a concept of finite freeness. 
Note that this is the completely opposite ordering of how one would typically develop free probability.
Rather than defining freeness and then investigating how operations act with respect to that definition, we defined a collection of operations and will try to characterize the property that leads to them.
This is because ``the property'' is no longer a universal trait and will depend on the operation in question.
This is both a blessing and a curse: on the one hand it means that we will not have the convenience of just calling things ``free'' and then doing arbitrary things to them. 
On the other hand, it will allow us to apply our theory to operators that satisfy much a weaker constraint than ``freeness''.
In particular, we will be able to ``instantiate'' finite freeness using 
(computable) matrices, something that was not possible in free probability. 
A consequence of this will be a collection of majorization relations that we 
prove in Section~\ref{sec:majorization}.

We should note that we will intentionally blur the lines between polynomials and (classes of) real symmetric matrices in this section.
Both will be associated to a multiset, either via its roots (in the case of a polynomial) or its eigenvalues (in the case of matrices).
In this regard, the convolutions defined in Section~\ref{sec:formulas} can be extended to multisets by operating on the monic polynomials with the elements of that multiset as roots.
Equivalently, it can be extended to a (rotation invariant) operation on (classes of) real symmetric matrices by operating on the characteristic polynomials of those matrices.
Because all of the convolutions preserve real rootedness \cite{conv}, it is plausible that they could coincide with classical matrix operations, and this will fuel our definition of finite freeness.

\begin{definition}
We will say two $\M \times \M$ real symmetric matrices $A$ and $B$ are {\em in finite free position} (or simply {\em $\M$-free}) if 
\[
\mydet{xI - yA - zB} 
= \mydet{xI - yA} \sqsum_{\M} \mydet{xI - zB}
\]
for all $y, z \in \R$.
\end{definition}
\begin{remark}
Because of the heightened role that additive convolution plays in probability theory (for example, in many of the limit theorems in Section~\ref{sec:applications}), we will take a noticably ``additive centric'' view.
This is done to avoid having to resort to conventions like ``finitely multiplicatively free,'' but whether such terms are inevitable is another question.
So, at least in regards to this paper, the term ``free'' will be used in association with an additive property.

Also worth noting is that, at least with respect to the additive convolution, our definition is stronger than it needs to be (requiring the equality to work on all real numbers $y$ and $z$ rather than just $y = z = 1$).
It turns out that the increased structure will make more sense when applied to the multiplicative convolution, but it would be interesting to explore the effects of weakening these conditions.
\end{remark}

To show that finite freeness exists, we will use some results in the theory of hyperbolic polynomials.
Given a vector $\vec{e}$, a homogeneous polynomial $p$ is said to be {\em 
hyperbolic with respect to} $\vec{e}$ if $p(\vec{e}) \neq 0$ and the univariate 
polynomial $\hat{p}(t) = p(\vec{x} - t \vec{e})$ has only real roots for all 
points $x$. 
The canonical example of a hyperbolic polynomial is the determinant acting on the space of real symmetric matrices: the determinant is hyperbolic with respect to the identity matrix since $\mydet{tI - X}$ is real rooted for all symmetric matrices $X$ (the roots then being simply the eigenvalues).
An extremely useful characterization of hyperbolic polynomials \cite{lax, HV}:

\begin{theorem}
\label{thm:HV}
A polynomial $p$ on $\R^3$ is hyperbolic of degree $\M$ with respect to the vector $\vec{e} = (1, 0, 0)$ and satisfies $p(\vec{e})=1$ if and only if there exist $\M \times \M$ real symmetric matrices $B, C$ such that
\[
p(x, y, z) 
= \mydet{xI + yB + zC}.
\]
\end{theorem}

A corollary to Theorem~\ref{thm:HV} is a rather strong statement concerning the existence of matrices in finite free position.

\begin{lemma}\label{lem:ff_position}
For any real symmetric matrices $A, B$ there exists a rotation matrix $R$ such that $A$ and $R^TBR$ are in finite free position.
\end{lemma}
\begin{proof}
Consider the polynomial
\[
p(x, y, z) 
= \mydet{xI - yA} \sqsum_{\M} \mydet{xI - zB}
\]
We first show that $p$ is hyperbolic with respect to the vector $\vec{e} = (1, 0, 0)$.

Clearly $p(1, 0, 0) \neq 0$ and so it remains to show that $\hat{p}(x_0 - t, y_0, z_0)$ is real rooted for all fixed values $x_0, y_0, z_0$.
Plugging in, we get
\begin{align}
\hat{p}(x_0 - t, y_0, z_0) 
&= \mydet{(x_0 - t)I - y_0A} \sqsum_{\M} \mydet{(x_0 - t)I - z_0B}
\notag \\&= (-1)^{\M} \mydet{tI + \hat{A}} \sqsum_{\M} (-1)^{\M} \mydet{tI + \hat{B}} \label{eq:real_roots}
\end{align} 
where 
\[
\hat{A} 
= y_0A - x_0I
\AND
\hat{B} 
= z_0B - x_0I
\]
are both Hermitian.
However, the fact that \ref{eq:real_roots} is always real rooted when $A$ and $B$ are Hermitian was shown in \cite{conv} (as noted in Section~\ref{sec:formulas}).

Hence $p$ meets the requirements for Theorem~\ref{thm:HV}, which means it can be written 
in the form 
\[
p(x, y, z) 
= \mydet{xI + yU + z V}
\]
for some matrices $U$ and $V$.
Since $p(x, 1, 0) = \mydet{xI - A}$, it must be that $\eigen{U} = \eigen{-A}$, and similarly, $\eigen{V} = \eigen{-B}$.
In particular there exists rotations $P$ and $Q$ such that 
\[
P^T U P 
= -A
\AND
Q^T V Q 
= -B
\]
Let $R = Q^T P$.
Then we have 
\begin{align*}
p(x, y, z) 
&= \mydet{xI + yU + z V}
\\&= \mydet{xI + yU - Q B Q^T}\mydet{P^TP}
\\&= \mydet{xI - yA - P^T Q B Q^T P}
\\&= \mydet{xI - yA - R^T B R}
\end{align*}
which is exactly the statement that $A$ and $R^TBR$ are in finite free position.
\end{proof}

Equation~(\ref{eq:add_conv}) then provides a collection of identities that characterize matrices in finite free position.
The decomposition of determinants into mixed discriminants shown by Equation~(\ref{eq:binomial}) allow us to state these identities explicitly:
\begin{lemma}
\label{lem:ff_equality}
For two $\M \times \M$ real symmetric matrices $A$ and $B$, the following are equivalent:
\begin{enumerate}
\item $A$ and $B$ are in finite free position
\item $\M!~\D{A[j], B[i], I[\M-j-i]} = \D{ A[j], I[\M-j]}\D{B[i], I[\M-i]}$ for all $i, j$.
\end{enumerate}
\end{lemma}
\begin{proof}
Simply expand the relation 
\[
\mydet{xI - yA - zB} 
= \mydet{xI - yA} \sqsum_{\M} \mydet{xI - zB}
\]
using Equation~(\ref{eq:add_conv}) and equate coefficients.
\end{proof}

In general, the rotation guaranteed by Lemma~\ref{lem:ff_position} may not be unique.
Since determinants are invariant under rotations, we have
\[
p(x, y, z) 
= \mydet{xI - yA - zB} 
= \mydet{R^T(xI - yA - zB)R} 
= \mydet{xI - yR^TAR - zR^TBR}
\]
and so any rotation $R$ that leaves $A$ unchanged provides a new rotation of $B$ that is in finite free position with $A$.
The extreme case of this is when $A = I$ (the identity), which is in finite free position with every matrix.
Since $I$ is freely independent (in the Voiculescu sense) from all operators, this should be of little surprise.
However, this fact has useful consequences:

\begin{corollary}
\label{cor:ff_translate}
If $A$ and $B$ are in finite free position, then $sA + tI$ and $uB + vI$ are in finite free position for any $s, t, u, v \in \R$.
\end{corollary}
\begin{proof}
Follows directly from the definition.
\end{proof}

Despite the lack of uniqueness, it would be useful to characterize any property of finite free position that can be guaranteed (either by all or by some rotation).
The next lemma is a step in that direction:

\begin{lemma}\label{lem:ff_position2}
Let $A$ and $B$ be $\M \times \M$ real symmetric matrices with $A$ diagonal.
Then there exists a rotation matrix $R$ such that $A$ and $R^TBR$ are in finite free position \textbf{and} the diagonal of $R^TBR$ is constant.
\end{lemma}
\begin{proof}
By Corollary~\ref{cor:ff_translate}, the statement is true for all pairs $(A, B)$ if and only if is true for all pairs $(A, B - \tr{B}I)$, so without loss of generality, we can assume $\tr{B} = 0$.
Let $Q$ be the rotation guaranteed by Lemma~\ref{lem:ff_position}.
For any matrix $X$, we have the relation $\D{X, I[\M-1]} = \M!\tr{X}$, so Lemma~\ref{lem:ff_equality} implies
\[
\M!~\D{A[j], Q^TBQ, I[\M-j-1]} 
= \D{ A[j], I[\M-j]} \D{Q^TBQ, I[\M-1]} 
= 0
\]
for all $j$.
In particular, taking various linear combinations of these gives
\[
\D{(x I - A)[\M-1], Q^TBQ} 
= 0
\]
for all values of $x$.
Letting $p(x) = \mydet{xI - A}$, this gives
\[
p(x) \sum_i \frac{b_i}{x - a_i} 
= 0
\]
where $b_i$ is the $i$th diagonal entry of $R^TBR$ (and same for $a_i$).
Since the set of values for which $p(x) = 0$ is finite, continuity implies that
\[
\sum_i \frac{b_i}{x - a_i} 
= 0
\]
for all $x$.

Note that for each distinct value of $u \in \{ a_i \}$, we can choose $x$ sufficiently close to $u$ such that 
the terms with $1 / (x- u)$ dominate the other terms.
Thus for the sum to be $0$, it must be that the set of $b_i$ for which $a_i = u$ sum to zero (in the case that there is a unique $a_i$, then the corresponding $b_i$ must equal $0$).

Recall from the previous discussion that, given $A$ and $B$ in finite free position, we have that $A$ and $S^T B S$ are in finite free position for any rotation $S$ such that $S^TAS = A$.
This is exactly the set of rotations that act independently on the eigenspaces of $A$, which (because $A$ Is diagonal) are the set of rotations that act on the submatrices of $A$ for which the diagonal entries are the same.
Thus we can pick a rotation $S$ that averages the elements of the diagonal of $B$ corresponding to a single eigenspace of $A$.
Thus the rotation $R = QS$ places $A$ and $R^TBR$ in finite free position and causes $R^TBR$ to have a $0$ diagonal.
\end{proof}

We now observing the analogous statement to Lemma~\ref{lem:ff_equality} for the multiplicative convolution:
\begin{lemma}
\label{lem:ff_mult_equality}
For two $\M \times \M$ real symmetric matrices $A$ and $B$, the following are equivalent:
\begin{enumerate}
\item $\mydet{xI - AB} = \mydet{xI - A} \sqmult_{\M} \mydet{xI - B}$
\item $\M!~\D{AB[i], I[\M-i]} = \D{ A[i], I[\M-i]} \D{B[i], I[\M-i]}$ for all $i$.
\end{enumerate}
\end{lemma}
\begin{proof}
Simply expand the relation 
\[
\mydet{xI - AB} 
= \mydet{xI - A} \sqmult_{\M} \mydet{xI - B}
\]
using Equation~(\ref{eq:mult_conv}) and equate coefficients.
\end{proof}

Hence we have the following relationship between finite freeness and the multiplicative convolution:
\begin{corollary}
\label{lem:ff_mult}
If $B$ is invertible and $A$ and $B^{-1}$ are in finite free position, then 
\[
\mydet{xI - AB} 
= \mydet{xI - A} \sqmult_{\M} \mydet{xI - B}
\]
\end{corollary}
\begin{proof}
For $A$ and $B^{-1}$ in finite free position, we have by Lemma~\ref{lem:ff_equality}
\[
\M!~\D{A[i], B^{-1}[\M-i]} 
= \D{A[i], I[\M-i]}\D{B^{-1}[\M-i], I[i]}
\] 
for all $i$.
Multiplying both sides by $\mydet{B}$ gives
\[
\M!~\D{AB[i], I[\M-i]} 
= \D{A[i], I[\M-i]} \D{B[i], I[\M-i]}
\]
which is precisely what is needed by Lemma~\ref{lem:ff_mult_equality}.
\end{proof}

\subsection{Comparison to Voiculescu's freeness}

To compare finite freeness with Voiculescu's version, it is instructive to see what is implied by the latter which is not implied by the former.
For example, if operators $A$ and $B$ are freely independent (in the Voiculescu sense), then 
\begin{enumerate}
\item $\mu_A \boxplus \mu_B = \mu_{A + B}$
\item $\mu_A \boxtimes \mu_B = \mu_{AB}$
\item $f(\mathcal{A})$ and $g(\mathcal{B})$ are freely independent for any functions $f, g$.
\end{enumerate}
The first statement is true for finite freeness (by construction), but the second and third are not true in general.
This should not come as a huge surprise, as the requirements for a matrix to instantiate additive convolution (Lemma~\ref{lem:ff_equality}) and multiplicative convolution (Lemma~\ref{lem:ff_mult_equality}) are different.
Furthermore, different instantiations of finite freeness can give different results when multiplied, as the next example shows:

\begin{example}
Consider 
\[
A = \left(
\begin{array}{ccc}
1 & 0 & 0 \\
0 & 2 & 0 \\
0 & 0 & 3 
\end{array} \right)
\AND
B = \left(
\begin{array}{ccc}
2 & 0 & 1 \\
0 & 2 & 0 \\
1 & 0 & 2 
\end{array} \right), 
\]
both of which have characteristic polynomial $p_A(x) = p_B(x) = x^3 - 6x^2 + 11x - 6$.
On one hand, we have
\[
[p_A \sqsum_{3} p_B](x) 
= x^3 - 12x^2 + 46x - 56
\AND
[p_A \sqmult_{3} p_B](x) 
= x^3 - 12x^2 + \frac{121}{3}x - 36
\]
while on the other hand, we have
\[
\mydet{xI - A - B} 
= x^3 - 12x^2 + 46x - 56
\AND
\mydet{xI - AB} 
= x^3 - 12x^2 + 41x - 36
\]
Hence $A$ and $B$ are in free position, but do not instantiate the multiplicative convolution.
We could also have taken 
\[
B^+ = \left(
\begin{array}{ccc}
2 & \sqrt{2}/2 & 0 \\
\sqrt{2}/2 & 2 & \sqrt{2}/2 \\
0 & \sqrt{2}/2 & 2 
\end{array} \right)
\qquad
\text{or}
\qquad
B^- = \left(
\begin{array}{ccc}
2 & \sqrt{2}/2 & 0 \\
\sqrt{2}/2 & 2 & -\sqrt{2}/2 \\
0 & -\sqrt{2}/2 & 2 
\end{array}\right)
\]
which give
\[
\mydet{xI - A - B^+} = \mydet{xI - A - B^-} = x^3 - 6x^2 + 11x - 6
\]
and
\[
\mydet{xI - AB^+} = \mydet{xI - AB^-} = x^3 - 12x^2 + 40x - 36
\]
\end{example}

On the other hand, there are special cases where the additive and multiplicative convolutions do coincide.
One such situation occurs when one of the matrices is a projection.
The proofs are heavy in computation and are not particularly enlightening, so we have separated them out into Section~\ref{sec:comp}.
The result of these computations, however, leads to the following corollary:
\begin{corollary}
If $A$ and $B$ are $\M \times \M$ real symmetric matrices in finite free position with $A^2 = A$ then 
\[
\mydet{xI - AB} 
= \mydet{xI - A} \sqmult_{\M} \mydet{xI - B}
\]
\end{corollary}
\begin{proof}
The fact that $A$ and $B$ are in finite free position and $A^2 = A$ are exactly the hypotheses in Lemma~\ref{lem:a2}.
Hence we have
\[
\M!~\D{A[i], B[j], AB[k], I[\M-i-j-k]} 
= \D{ A[i+k], I[\M-i-k]} \D{B[j+k], I[\M-j-k]}.
\]
for all $i, j, k$ with $i + j + k \leq \M$.
The case $i = j = 0$, in particular, gives
\[
\M!~\D{AB[k], I[\M-k]} 
= \D{ A[k], I[\M-k]} \D{B[k], I[\M-k]}
\]
for all $k$, which by Lemma~\ref{lem:ff_mult_equality} suffices to prove the lemma.
\end{proof}

\subsection{Application: Majorization}\label{sec:majorization}

\newcommand{\addc}[2]{\Delta_{#2}(#1)}
\newcommand{\mytr}[1]{\mathrm{Tr}\left[ #1 \right]}
\newcommand{\trans}{T}

The fact that one can find matrices in finite free position is far more useful 
than simply for the ability to ``instantiate'' the operation with fixed 
matrices.
In particular, we will use the characterization in Lemma~\ref{lem:ff_position2} 
to show that the roots of certain convolutions satisfy a {\em majorization} 
relation. 
Such a relation can be useful due to the large number of inequalities that it 
implies (see \cite{olkin}).
A collection $x_1 \geq \dots \geq x_n$ of real numbers is said to 
{\em majorize} the collection $y_1 \geq \dots \geq y_n$ if the inequality
\[
\sum_{i = 1}^k x_i \geq \sum_{i=1}^k y_i
\]
for all $k \leq n$ and (furthermore) holds with equality when $k = 
n$.
Perhaps the most well-known majorization relation (first noticed by Schur) 
occurs between the eigenvalues of a Hermitian matrix and the diagonal entries 
of that matrix.
Hence Lemma~\ref{lem:ff_position2} implies a similar relation between the 
roots of a polynomial and an additive convolution:
\begin{corollary}\label{cor:easy_majorization}
Let $p, q$ be a real rooted, degree $d$ matrices for which the sum of the roots 
of $q$ is $0$.
Then the roots of $[p \boxplus_d q]$ majorize the roots of $p$,
\end{corollary}

The remainder of this section will be dedicated to showing a generalization of 
Corollary~\ref{cor:easy_majorization}. 
Given an $m \times n$ matrix $A$, the {\em $k$th compound matrix of $A$} is the 
$\binom{m}{k} \times \binom{n}{k}$ matrix $C_k(M)$ consisting of the $k \times 
k$ minors of $A$.
That is, for $J \subseteq \binom{[m]}{k}$ and $K \subseteq \binom{[n]}{k}$, we 
have
\[
C_k(A)_{J, K} = \mydet{ A_{[J, K]} }
\]
where $A_{[J, K]}$ denotes the submatrix of $A$ consisting of rows in $J$ and 
columns in $K$.
One can then define the {\em $k$th additive compound matrix of $A$} as
\[
\addc{A}{k} = \frac{\partial}{\partial t} C_k(I + t A) ~\bigg|_{t=0}.
\]

The following properties of additive compound matrices are well known (see 
\cite{fiedler}, for example):
\begin{lemma}\label{lem:additive_compound}~
\begin{enumerate}
\item If $A$ and $B$ are $m \times n$ matrices, then 
\[
\addc{A}{k} + \addc{B}{k} = \addc{A+B}{k}.
\]
\item If $A$ is a $\M \times \M$ diagonal matrix, then $\addc{A}{k}$ is a 
diagonal matrix.
\item If $A$ is a $\M \times \M$ Hermitian matrix, then $\addc{A}{k}$ is a 
Hermitian matrix.
\item $\addc{A}{k}_{J, J} = \mytr{A_{[J, J]}}$ for all $J \in \binom{[\M]}{k}$
\item If $A$ is a $\M\times \M$ matrix with eigenvalues $\lambda_1, \dots, 
\lambda_m$.
Then the eigenvalues of $\addc{A}{k}$ are
\[
\left\{ \sum_{i \in S} \lambda_i : S \in \binom{[\M]}{k} \right\}.
\]
\end{enumerate}
\end{lemma}

The utility of the additive compound matrices (at least as it applies to 
majorization) can be seen in the last property listed in 
Lemma~\ref{lem:additive_compound}: it allows one to turn statements regarding 
the sum of the $k$ largest eigenvalues of a matrix $A$ into a statement 
regarding the (single) largest eigenvalue of $\addc{A}{k}$.
The statement regarding the (single) largest eigenvalue that will be of 
interest to us is proved in the next lemma:

\begin{lemma}\label{lem:increaset}
Let $A$ be a $\M\times \M$ real, diagonal matrix and let $B$ be a $\M \times 
\M$ real symmetric matrix with $0$ diagonal.
Then the function 
\[
f(t) = \lambda_{\max}(A + t B).
\]
is increasing for $t \geq 0$.
\end{lemma}
\begin{proof}
For each value of $t$, let $v_t$ be a maximal eigenvector, so that 
\[
f(t) 
= v_t^\trans (A + t B) v_t 
= \max_{\|v\| = 1 } v^\trans (A + t B) v.
\]
Then for $s \neq t$, we have
\[
f(s) 
\geq  v_t^\trans (A + s B) v_t = \max_{\|v\| = 1 } v^\trans (A + t B) v.
= f(t) + (s-t) v_t^T B v_t^T
\]
and so for $s > t$ we have
\[
v_t^\trans B v_t \leq \frac{f(s) - f(t)}{s-t} \leq v_s^\trans B v_s.
\]
where it is easy to check that $v_0^\trans B v_0 = 0$ (since $v_0$ is an 
elementary basis vector).
Hence for $s > t$, we have
\[
0 \leq v_t^\trans B v_t \leq \frac{f(s) - f(t)}{s-t}
\]
and so $f(t)$ is increasing.
\end{proof}

We are now in position to show the generalized result:

\begin{theorem}\label{thm:majorization}
Let $p, q$ be degree $\M$ polynomials with the sum of the roots of $q$ being 
$0$.
For $t \geq 0$, define the polynomials
\[
r_t(x) = [ p(x) \boxplus_\M t^\M q(x/t) ].
\] 
Then the roots of $r_t(x)$ majorize the roots of $r_s(x)$ if and only if $t 
\geq s$.
\end{theorem}
\begin{proof}
By Lemma~\ref{lem:ff_position2}, there exist real symmetric matrices $A, B$ 
with $A$ a diagonal matrix and $B$ with $0$ diagonal for which 
\[
p(x) = \mydet{x I - A}
\AND
q(x) = \mydet{x I - B}
\AND
r_t(x) = \mydet{x I - A - t B}.
\]
Note that by Lemma~\ref{lem:additive_compound}, we have that 
\begin{itemize}
\item $\addc{A}{k}$ is a diagonal matrix for all $k$
\item $\addc{B}{k}$ has $0$ diagonal for all $k$
\end{itemize}
and so by Lemma~\ref{lem:increaset}, we have that for $t \geq s$
\[
\lambda_{\max}(  \addc{A + t B}{k} ) \geq \lambda_{\max}(  \addc{A + s B}{k} ). 
\]
In other words, the sum of the largest $k$ roots of $r_t(x)$ is at least as 
large as the sum of the largest $k$ roots of $r_s(x)$.
It is easy to check that this holds with equality when $k = n$, and so the 
roots of $r_t(x)$ majorize the roots of $r_s(x)$ (by definition).
\end{proof}

Note that $r_0(x) = p(x)$ so Theorem~\ref{thm:majorization} contains 
Corollary~\ref{cor:easy_majorization} as a special case.
It is easy to see, however, that one cannot derive 
Theorem~\ref{thm:majorization} from Corollary~\ref{cor:easy_majorization} as 
the next example shows:

\begin{example}
Consider the polynomial $q(x) = x^4 - 12 x^2$ so that for any real rooted $p$, 
we have
\[
[ p(x) \boxplus_4 q(x) ] = p(x) - p''(x)
\AND
[ p(x) \boxplus_4 2^4 q(x/2) ] = p(x) - 4 p''(x)
\]
and so Theorem~\ref{thm:majorization} implies that the roots of $p(x) - p''(x)$ 
majorize the roots of $p(x) - 4 p''(x)$.
In order to derive a similar result directly from 
Corollary~\ref{cor:easy_majorization}, we would need a polynomial $r(x)$ for 
which
\[ 
[ r(x) \boxplus_4 q(x) ] = 2^\M q(x/2)
\] 
and one can check that the only such polynomial is
\[
r(x) = x^4 - 36 x^2 - 72
\]
which is not real rooted (so Corollary~\ref{cor:easy_majorization} would not 
apply).
\end{example}


\subsection{Computations}\label{sec:comp}

The goal of this section is to prove Lemma~\ref{lem:a2} below.
We first prove a combinatorial identity that is an easy consequence of the extension to the binomial theorem:
\[
(x+1)^{-k} 
= \sum_i \binom{-k}{i}x^i 
= \sum_i (-1)^i \binom{k+i-1}{i}.
\]
\begin{lemma}\label{lem:binomial}
Let $j,k, n$ be nonnegative integers such that $j+k \leq n$.
Then
\[
\sum_{t=0}^k (-1)^{t} \binom{k}{t}\binom{n-t}{n-j-k-t} 
= \binom{n-k}{n-j-k}
\]
\end{lemma}
\begin{proof}
We expand 
\[
(x+1)^k 
= \sum_{i} x^i \binom{k}{i}
\]
and 
\[
(x+1)^{-j-k-1} 
= \sum_{i} (-1)^i x^i \binom{j+k+i}{i}.
\]
Hence the coefficient of $x^{n-j-k}$ in the product is
\[
\sum_{t} \binom{k}{t} (-1)^{n-j-k-t} \binom{n-t}{n-j-k-t}.
\]
On the other hand, we can multiply first and then expand to get
\[
(x+1)^{-j-1} 
= \sum_i (-1)^i x^i \binom{j+i-1}{i}
\]
so that the coefficient of $x^{n-j-k}$ is 
\[
(-1)^{n-j-k} \binom{n-k}{n-j-k}.
\]
The lemma then follows by equating the two representations for the same coefficient.
\end{proof}

We now have the tools to prove the lemma.
For $\M  \times \M$ matrices, we define the quantity 
\[
f(i, j, k) 
= \D{A[i], B[j], AB[k], I[\M-i-j-k]}
\]
\begin{lemma}\label{lem:a2}
Let $A$ and $B$ be $\M \times \M$ matrices that satisfy the identities
\begin{equation}\label{eq:relation}
\M!~f(i, j, 0) 
= f(i, 0, 0) f(0, j, 0)
\end{equation}
for all $i, j$ and such that $A^2 = A$.
Then  
\begin{equation}\label{eq:to_prove}
\M!~f(i, j, k) 
= f(i+k, 0) f(0, j+k, 0)
\end{equation}
\end{lemma}
\begin{proof}
Let $\Tr{A} = a$.
Without loss of generality, we can assume $A$ is diagonal (with $a$ $1$s and $(\M - a)$ $0$s on the diagonal).
Hence we can use the formula in Equation~(\ref{eq:permanent}) to calculate $\D{A[i], I[\M-i]}$ by counting the number of permutations for which the first $i$ elements are at most $a$.
That is, 
\begin{equation}
\label{eq:as}
f(i, 0, 0) 
= (a)(a-1) \dots (a-i+1)(\M-i) \dots (1) = \frac{a!(\M-i)!}{(a-i)!}.
\end{equation}
Hence Equation~(\ref{eq:to_prove}) is equivalent to 
\begin{equation}\label{eq:to_show}
f(i, j, k) 
= f(0,j+k, 0) \frac{a!(\M-i-k)!}{\M!(a-i-k)!}.
\end{equation}
Using Lemma~\ref{lem:trace_formula}, we have the relation
\[
af(i, j, k) 
= i f(i, j, k) + j f(i, j-1, k+1) + kf(i, j, k) + (\M-i-j-k) f(i+1, j, k), 
\]
which, after rearranging and shifting indices ($j \to j+1, k \to k-1$) gives
\begin{equation}\label{eq:fg}
(j+1) f(i, j, k) 
= -(a-i-k+1) f(i, j+1, k-1) + (\M-i-j-k) f(i+1, j+1, k-1).
\end{equation}
If we define the quantity $g(i, j, k)$ as 
\[
f(i, j, k) 
= (-1)^i\frac{j!(\M-i-j-k)!}{(a - i - k)!} g(i, j, k)
\]
then substituting gives
\begin{equation}\label{eq:recursion}
g(i, j, k) 
= g(i, j+1, k-1) + g(i+1, j+1, k-1).
\end{equation}
We claim that Equation~(\ref{eq:recursion}) implies  
\begin{equation}\label{eq:solve_recursion}
g(i, j, k) 
= \sum_{t=0}^k \binom{k}{t}g(i + t, j + k, 0).
\end{equation}
This is trivially true for the case $k = 0$, so we can assume it to be true for $k = K-1$ with $K > 1$ and consider the case $k = K$.
Using the relation, we have
\[
g(i, j, k) = g(i, j+1, k-1)  + g(i+1, j+1, k-1)
\]
which by the inductive hypothesis means
\begin{align*}
g(i, j, K) 
&= \sum_{t=0}^{k-1} \binom{k-1}{t} g(i + t, j + k, 0) + \sum_{t=0}^{k-1} \binom{k-1}{t} g(i + t + 1, j + k, 0)
\\&= \sum_{t=0}^{k-1} \binom{k-1}{t} g(i + t, j + k, 0) + \sum_{t=1}^{k}  \binom{k-1}{t-1} g(i + t, j + k, 0)
\\&= \sum_{t=0}^{k} g(i + t, j + k, 0)
\end{align*}
where the last equality uses the identity $\binom{n}{k} = \binom{n-1}{k-1} + \binom{n-1}{k}$.
Combining Equations~(\ref{eq:fg}) and (\ref{eq:solve_recursion}) give
\[
f(i, j, k) = \sum_{t=0}^k \binom{k}{t} (-1)^{t}\frac{j!(\M-i-j-k)!(a - i - t)!}{(j+k)!(\M-i-j-k-t)!(a - i - k)!} f(i+t, j+k, 0).
\]

Now Equations~(\ref{eq:relation}) and (\ref{eq:as}) imply
\[
f(i + t, j+k, 0) = \frac{1}{\M!}f(i + t, 0, 0)f(0, j+k, 0) = \frac{a!(\M-i-t)!}{\M!(a-i-t)!}f(0, j+k, 0)
\]
so that 
\begin{align*}
f(i, j, k)
&= a!f(0, j+k, 0) \sum_{t=0}^k \binom{k}{t} (-1)^{t}\frac{j!(\M-i-j-k)!(\M - i - t)!}{\M!(j+k)!(\M-i-j-k-t)!(a - i - k)!}
\\&= f(0, j+k, 0)\frac{a!j!(\M-i-j-k)!}{\M!(a-i-k)!} \sum_{t=0}^k (-1)^{t} \binom{k}{t}\binom{\M-i-t}{\M-i-j-k-t}.
\end{align*}
Using Lemma~\ref{lem:binomial} reduces this to
\[
f(i, j, k) 
= f(0, j+k, 0)\frac{a!j!(\M-i-j-k)!}{\M!(a-i-k)!} \frac{(\M-i-k)!}{(\M-i-k-j)!j!} 
= f(0, j+k, 0)\frac{a!(\M-i-k)!}{\M!(a-i-k)!}
\]
which equals (\ref{eq:to_show}) as required.
\end{proof}

\section{Applications}\label{sec:applications}

\renewcommand{\deriv}[1]{ D_{#1}}

In this final section, we explore some of the consequences of the definition of the $\M$-finite R-transform.
There are number of ``free'' versions of classical distributions, for example the role of ``free Gaussian'' is played by the semicircle law and the role of ``free Poisson'' distribution is played by the Marchenko--Pastur law \cite{MP}.
The goal then is to try to derive the finite free version of such distributions and then to show they behave (with respect to the free additive convolution) in the way they should.
In particular, we will focus on additive limit theorems associated with these distributions, such as the central limit theorem and Poisson limit theorem.
To do so, we will need to compute additive convolutions, and so we begin by proving some computational properties.

\subsection{Properties of the symmetric additive convolution}
\label{sec:add_prop}

Recall from the Section~\ref{sec:formulas}, that we have the following formula for the additive convolution of two degree $\M$ polynomials:
\begin{equation}
\label{eq:add_conv2}
[p \sqsum_{\M} q](x) 
= \sum_{i + j \leq \M} x^{\M-i-j}(-1)^{i+j} \frac{(\M-i)!(\M-j)!}{(\M-i-j)!\M!} p_i q_j.
\end{equation}
While Equation~(\ref{eq:add_conv2}) was (in that case) only defined for polynomials of degree $\M$, such a formula can be applied to any polynomials.
In particular, we will consider the case when $p$ and $q$ have degree {\em at most} $\M$.

Three useful observations can be made directly from the formula.
The first is linearity:
for any polynomials $p,q,r$ and any constant $\alpha$, we have
\[
[p \sqsum_{\M} (\alpha q + r)] 
= \alpha [p \sqsum_{\M} q] + [p \sqsum_{\M} r] 
\]
The second is the observation that when $\deg(f) = \M$, we have $[f \sqsum_{\M} x^{\M}](y) = f(y)$.
Third is the observation that $[p \sqsum_{\M} q] = [q \sqsum_{\M} p]$.

We now write the formula in Equation~(\ref{eq:add_conv}) in a slightly different form:
\begin{equation}
\label{eq:add_conv_d}
[p \sqsum_{\M} q](x) 
= \frac{1}{\M!}\sum_{i=0}^{\M} p^{(i)}(x)q^{(\M-i)}(0)
\end{equation}
where $p^{(i)}$ denotes the $i$th derivative of $p$.
In this form, the following lemma is almost immediate:
\begin{lemma}
\label{lem:deriv_commute}
Let $R =  \sum_i a_i \ideriv{x}{i}$ be a linear differential operator.
Then 
\[
[R \{ p \} \sqsum_{\M} q](x) 
= [p \sqsum_{\M} R \{ q \} ](x) = R \{ [p \sqsum_{\M} q](x) \}
\]
\end{lemma}
\begin{proof}
By linearity, it suffices to prove this for the operator $R = \ideriv{x}{k}$.
\[
\left[ p^{(k)} \sqsum_{\M} q \right](x)
= \frac{1}{\M!}\sum_{i=0}^{\M} p^{(i+k)}(x) q^{(\M-i)}(0)
= \ideriv{x}{k} \frac{1}{\M!}\sum_{i=0}^{\M} p^{(i)}(x)q^{(\M-i)}(0) 
= \ideriv{x}{k} \left[ p \sqsum_{\M} q\right](x).
\]
By the same argument, 
\[
\ideriv{x}{k} \left[ p \sqsum_{\M} q \right](x) 
= [q^{(k)} \sqsum_{\M} p](x) 
= [p \sqsum_{\M} q^{(k)}](x) 
\]
where the last equality is due to the commutativity that was observed earlier.
\end{proof}

Lemma~\ref{lem:deriv_commute} gives an effective way to compute the symmetric additive convolution.
As an example, we give the following corollary proving that, in the space of degree $\M$ polynomials, the symmetric additive convolution is invertible:
\begin{corollary}
For any degree $\M$ polynomial $p$, there exists a degree $\M$ polynomial $q$ such that $[p \sqsum_{\M} q] = x^{\M}$.
\end{corollary}
\begin{proof}
Since $p$ has degree $\M$, we can write  $p(x) = R \{ x^{\M}\}$ where $R$ is a linear differential operator with a nonzero identity term.
Viewing this as a power series, we can compute the formal (multiplicative) inverse $R^{-1}$ and then compute $q = R^{-1} \{ x^{\M}\}$.
As a result, we have
\[
[p \sqsum_{\M} q](x) 
= [R\{ x^d \} , R^{-1} \{ x^d \} ](x) = R\{ R^{-1}\{ [ x^{\M} \sqsum_{\M} x^{\M} ] \} \} = x^{\M}
\]
\end{proof}

Last but not least, it is not hard to show that $\sqsum_{\M}$ is associative (and therefore forms an algebra on polynomials degree $\M$).

\begin{lemma}
Let $p, q, r$ be polynomials.
Then 
\[
[p \sqsum_{\M} [ q \sqsum_{\M} r]] 
= [[p \sqsum_{\M}  q] \sqsum_{\M} r]
\]
\end{lemma}
\begin{proof}
A quick calculation shows that both are equal to 
\[
\frac{1}{(\M!)^2}\sum_{i=0}^{\M} \sum_{j = 0}^{\M-i} f^{(i)}(x)g^{(j)}(0)h^{(2\M-i-j)}(0)
\]
\end{proof}
The algebra formed by $\sqsum_{\M}$ is isomorphic to the algebra of truncated polynomials $\C[x] / \langle x^{\M+1} \rangle$ under multiplication.
This can easily be seen by considering the homomorphism $\phi$ taking a polynomial $p$ to the differential operator $R$ such that $R \{ x^{\M} \} = p$.
As noted in \cite{conv}, the symmetric additive convolution of two real rooted polynomials is another real rooted polynomial.
This forms a cone in the algebra of truncated polynomials that could be interesting in its own right.

\subsection{Distributions and Limit Theorems}

The approach we will take to finding the finite free versions of various distributions is to reverse engineer them from their R-transforms.
That is, for a given R-transform, we would like to find the degree $\M$ polynomial whose $\M$-finite R-transform matches the given R-transform (on the first $\M$ coefficients).
The introduces a slight technicality; the reverse procedure is not unique, since any constant multiple of a polynomial has the same roots that the original polynomial did.
Hence we will introduce the notation $p \approx q$ to denote that $p$ and $q$ have the same roots (or $p = cq$ for some $c \neq 0$, if you prefer).
The next lemma comes directly from the definitions in 
Section~\ref{sec:add_conv}:

\begin{lemma}
\label{lem:halfway}
Let $Q(x)$ be a polynomial. 
Then
\[
Q\left(\frac{\deriv{x}}{\M} \right) \left\{ x^{\M} \right\} 
\approx \mydet{xI - A}
\iff
\frac{1}{\M} \frac{Q'(x)}{Q(x)} 
\equiv -\myrtrans{\M}{\mu_A}{x} \mymod{x}{\M+1}
\]
\end{lemma}
\begin{proof}
By Corollary~\ref{cor:rtrans}, we have 
\[
\myrtrans{\M}{\mu_A}{s} 
\equiv -\frac{1}{m}\deriv{s} \ln \E{}{e^{-{\M}sT_A}} \mymod{s}{\M}
\]
and so 
\[
\frac{1}{\M} \frac{Q'(x)}{Q(x)} 
\equiv -\myrtrans{\M}{\mu_A}{x} \mymod{x}{\M+1}
\]
if and only if
\[
Q(x) 
\approx \E{}{e^{-{\M}sT_A}} \mymod{x}{\M+1}
\]
where $T_A$ is the \marcus~transform of $\eigen{A}$.
We note that for any scalar $t$ and any polynomial $p(x)$, we have by Taylor's theorem
\[
p(x + t) = \sum_{i} \frac{t^i}{i!}\ideriv{x}{i} \{ p(x) \} = e^{t\deriv{x}} \{ p(x) \}.
\] 
Hence
\[
Q\left(\frac{\deriv{x}}{\M} \right) \left\{ x^{\M} \right\} 
\approx \E{}{e^{-T_A\deriv{x}}} \left\{ x^{\M} \right\} 
= \E{}{(x - T_A)^{\M}} 
\]
which is $\mydet{xI - A}$ by definition of $T_A$.
\end{proof}

It is worth noting that, for our purposes, we do not need to worry about the restriction to $\M$ coefficients, as this comes automatically from taking derivatives of $x^{\M}$.
We also note that the constant of integration (that will come when we integrate) is precisely the cause of the ambiguity that necessitated the use of $\approx$ (and so can safely be ignored).

The rest of this section will be the derivation of polynomials corresponding to known distributions and then proofs of their associated limit laws.
The proofs will use the polynomials directly (and the techniques of Section~\ref{sec:add_prop}) and not the R-transforms --- something that is not possible in traditional free probability.
Of course (in these cases) we know the R-transform, so it is somewhat backwards to find the polynomials and use them when we could simply use the R-transforms.
The goal, however, is to prove that such a proof is possible, as there are (many) instances where one can not compute the R-transform from a given distribution, and so we anticipate that the techniques displayed below will serve useful in such scenarios. 

\subsubsection{Constant}
The constant random variable is characterized by having first cumulant nonzero and all other cumulants zero, so that 
\[
\myrtrans{\M}{\mu_A}{s} = \mu
\]
for $\mu$ a constant.
Using Lemma~\ref{lem:halfway}, we can compute
\[
\ln Q(x) 
= -\mu\M \int \d{x} 
= -x\mu\M + c
\]
for some constant $c$.
Hence
\[
Q\left(\frac{\deriv{x}}{\M} \right) \left\{ \frac{x^{\M}}{\M!} \right\} 
\approx e^{-\mu\deriv{x}} \left\{ x^{\M} \right\} 
= (x - \mu)^{\M}.
\]
This is precisely what we would hope, as 
\[
(x - \mu)^{\M} 
= \mydet{xI - \mu I}
\]
which is exactly the random matrix version of adding a constant random variable.

The limit theorem associated with the constant distribution is the {\em law of large numbers}.
The classical version of this theorem states that if $X_i$ are independent random variables with $\E{}{X_i} = \mu$, then 
\[
\frac{1}{n} \sum_{i=0}^{n} X_i 
\xrightarrow{n \to \infty} \mu
\]
almost surely.  
Here we give an $\M$-finite version of this result: 
\begin{theorem}[Law of large numbers]
\label{thm:lln}
Let $p_1, p_2, \dots$ be a sequence of degree $\M$ real rooted polynomials with 
\[
p_i 
= \prod_j (x - r_{i,j})
\AND
\frac{1}{\M}\sum_j r_{i,j} =\mu
\]
for all $i$.
Assume further that 
\[
\frac{1}{\M}\sum_j r_{i,j}^2 < C
\]
for some constant $C$.
Then
\[
\lim_{n \to \infty} [q_1 \sqsum_{\M} \dots \sqsum_{\M} q_n] 
= (x - \mu)^{\M}
\]
where $q_i(x) = n^{-\M} p_i(n x)$.
\end{theorem}
\begin{proof}
For fixed $i$, we can write $p_i(x) = x^{\M} + a_1 x^{\M-1} + \dots$.
Since $a_1$ is the sum of the roots, we have (by the hypotheses) $a_1 = \M\mu$.
So if $P_i$ is the linear differential operator such that $P_i \{ x^{\M} \} = p_i(x)$, we have
\[
P_i 
= 1 + \frac{a_1}{\M}\deriv{x} + \dots 
= 1 + \mu\deriv{x} + \dots 
\]
Now let $Q_i$ be the differential operator such that $Q_i \{ x^{\M} \} = q_i(x)$
Then one can check that 
\[
Q_i 
= 1 + \frac{a_1}{{\M}n}\deriv{x} + O\left(n^{-2}\right) 
= 1 -  \frac{\mu}{n}\deriv{x} + O\left(n^{-2}\right)
\]
and so 
\[
[q_1 \sqsum_{\M} \dots \sqsum_{\M} q_n] 
= \left(\prod_{i=1}^n Q_i \right) \left\{ x^{\M} \right\} 
= \left(1 -  \frac{\mu}{n}\deriv{x} + O\left(n^{-2}\right)\right)^n \left \{ x^{\M} \right\}
\]
which converges to 
\[
e^{-\mu\deriv{x}} \left\{  x^{\M} \right\} 
= (x - \mu)^{\M}
\]
as $n \to \infty$.
\end{proof}

\subsubsection{Gaussian}
\label{sec:gaussian}

The Gaussian random variable is characterized by having first two cumulants nonzero and all other cumulants zero.
Given that we know how the constant random variable behaves, it suffices to consider the case when only the second cumulant is nonzero, so that 
\[
\myrtrans{\M}{\mu_A}{s} = s\theta
\]
for $\theta$ a constant.
The computation is similar to the constant random variable, as we get
\[
\ln Q(x) 
= -\theta\M \int x \d{x} 
= -\frac{\M \theta}{2} x^2 + c
\]
so that 
\[
Q\left(\frac{\deriv{x}}{\M} \right) \left\{ x^{\M}\right\} 
\approx e^{-\theta x^2 / (2 \M)} \left\{ x^{\M} \right\}
\approx H_{\M} \left( \sqrt{\frac{\M}{\theta}}x \right)
\]
where $H_{\M}(x)$ is a Hermite polynomial.
This leads to the following definition:
For real numbers $\mu$ and $\sigma$, let 
\begin{equation}
\label{eq:normal}
\mathfrak{N}_{\M}[\mu, \sigma^2](x) 
:= \left(\frac{\M-1}{\sigma^2}\right)^{-\M/2} H_{\M} \left( (x - \mu)\sqrt{\frac{\M-1}{\sigma^2}} \right)
\end{equation}

Firstly, we remark that the constant in front is merely to make the polynomial monic (it is irrelevant, since all that matters are the roots).
Secondly, we remark that it should come as no surprise that the corresponding polynomial should be the Hermite polynomial, since the root distribution must converge asymptotically to the semicircle law (the free probability version of the Gaussian), and it is well known that the roots of the Hermite polynomials do exactly that \cite{conv_hermite}.

Lastly, we want to comment on the factor of $\M-1$ (rather than $\M$ from the 
derivation above).
Recall that the coefficients of the $\M$-finite R-transform {\em converge} to the true R-transform.
Apart from the first coefficient, however, they do not match exactly.
In particular, an operator $X$ with mean and variance $\tau$ and $\sigma^2$ will have R-transform
\[
\rtrans{\mu_X}{s} 
= \tau + \sigma^2s + \dots
\]
and $\M$-finite R-transform
\[
\rtrans{\mu_X}{s} 
= \tau + \left(1 - \frac{1}{\M}\right)\sigma^2s + \dots
\]
Hence by working backwards from $\kappa_2 = 1$, we were mistakenly giving the distribution a variance of $\frac{\M}{\M+1}$.

\begin{remark}
Note that there are two distinct scalings of Hermite polynomials in the literature.
The first (which is used in \cite{szego_book}), defines Hermite polynomials as the polynomials which are orthogonal to the weight function $e^{x^2}$. 
The second (which we will use), defines Hermite polynomials as the polynomials which are orthogonal to the Gaussian weight function $e^{x^2/2}$. 
This seemed to be the more appropriate choice, given the role that the Hermite polynomials are playing in this context (as the ``Gaussian'' of finite free probability).
\end{remark}

The limit theorem associated with the Gaussian is the {\em central limit theorem}.
The classical version of this theorem says that for $X_i$ i.i.d. random variables $\E{}{X_i} = 0$ and $\mathrm{Var}[Xi] = \sigma^2 < \infty$, we have
\[
\frac{1}{\sqrt{n}}\sum_{i=1}^n X_i 
\xrightarrow{n \to \infty} N(0,\sigma^2).
\]
in distribution (where $N(\mu, \sigma^2)$ is the Gaussian distribution).
Here we give a proof in the $\M$-finite case:
\begin{theorem}[Central limit theorem]
\label{thm:CLT}
Let $p_1, p_2, \dots$ be a sequence of degree $\M$ real rooted polynomials with $p_i = \prod_j (x - r_{i,j})$ such that 
\[
\sum_j r_{i,j} 
= 0
\AND
\frac{1}{\M}\sum_j r_{i,j}^2  
= \sigma^2
\]
for all $i$.
Then
\[
\lim_{n \to \infty} [q_1 \sqsum_{\M} \dots \sqsum_{\M} q_n] 
= \mathfrak{N}_{\M}[0, \sigma^2]
\]
where $q_i(x) = n^{-\M/2} p_i(\sqrt{n} x)$.
\end{theorem}
\begin{proof}
For fixed $i$, we can write $p_i(x) = x^{\M} + a_1 x^{\M-1} + a_2 x^{\M-2} + \dots$.
Since $a_1$ is the sum of the roots and $a_2$ is the sum of the pairwise products of roots, we have (by the hypotheses)
\[
a_1 
= 0
\AND
a_1^2 - 2 a_2 
= \M \sigma^2
\]
so that $a_2 = -\M \sigma^2/2$.
So if $P_i$ is the linear differential operator such that $P_i \{ x^{\M} \} = p_i(x)$, we have
\[
P_i = 1 + \frac{a_1}{\M}\deriv{x} + \frac{a_2}{\M(\M-1)}\deriv{x}^2 + \dots  = 1 -  \frac{\sigma^2}{2(\M-1)}\deriv{x}^2 + \dots
\]
Now let $Q_i$ be the differential operator such that $Q_i \{ x^{\M} \} = q_i(x)$
The one can check that 
\[
Q_i 
= 1 + \frac{a_1}{\M\sqrt{n}}\deriv{x} + \frac{a_2}{n\M(\M-1)}\deriv{x}^2 + O\left(n^{-3/2}\right)  
= 1 -  \frac{\sigma^2}{2n(\M-1)}\deriv{x}^2 + O\left(n^{-3/2}\right)
\]
and so 
\begin{equation}
\label{eqn:sum_clt}
[q_1 \sqsum_{\M} \dots \sqsum_{\M} q_n] 
= \left(\prod_{i=1}^n Q_i \right) \left\{ x^{\M} \right\} 
= \left(1 -  \frac{\sigma^2}{2n(\M-1)}\deriv{x}^2 + O\left(n^{-3/2}\right) \right)^n \left \{ x^{\M} \right\}
\end{equation}
which one can check converges to 
\[
e^{-\sigma^2\deriv{x}^2 / 2(\M-1)} \left\{ x^{\M} \right\} 
= \mathfrak{N}_{\M}[0, \sigma^2]
\]
as $n \to \infty$.
\end{proof}

Typically one needs to assume bounds on the higher moments in order to guarantee that the higher order terms in Equation~(\ref{eqn:sum_clt}) converge to $0$ asymptotically.
For polynomials of a fixed degree, however, this is implied by the bound on the second moment and the finite support of the underlying distribution.

\subsubsection{Poisson}
\label{sec:poisson}
The Poisson random variable with parameter $\lambda$ is characterized by having all cumulants equal to $\lambda$, so that 
\[
\myrtrans{\M}{\mu_A}{s} = \frac{\lambda}{1 - s}
\]
This time, we have
\[
\ln Q(x) = -\lambda \M \int \frac{1}{1-x} \d{x} = \lambda\M \ln(1-x) + c 
\]
so that 
\[
Q\left(\frac{\deriv{x}}{\M} \right) \left\{ x^{\M}\right\} 
\approx \left( 1 - \frac{\deriv{x}}{\M} \right)^{\lambda \M} \left\{ x^{\M} \right\}
\approx L_{\M}^{((\lambda-1)\M)} (\M x)
\]
where $L_{\M}^{(\alpha)}(x)$ is an associated Laguerre polynomial.
\begin{equation}
\label{eq:laguerre}
L_n^{(\alpha)}(x) 
= \sum_i \frac{(-x)^i}{i!} \binom{n+\alpha}{n-i}.
\end{equation}
This leads to the following definition (we will assume $\lambda \M$ is an integer for simplicity):
\begin{equation}
\label{eq:poiss}
\mathfrak{P}_{\M}[\lambda] 
:= \M!(-\M)^{-\M} L_{\M}^{((\lambda-1)\M)} (\M x)
\end{equation}
where again the constant is only there so the resulting polynomial is monic.

Note that when $\lambda < 1$, we should expect to have an atom of probability $(1-\lambda)$ at $0$.
While this is not obvious in the current definition, the following Laguerre polynomial identity
\begin{equation}
\label{eq:lag_ident}
\frac{(-x)^i}{i!} L_n^{(i-n)}(x) = \frac{(-x)^n}{n!} L_i^{(n-i)}(x).
\end{equation}
can be applied to get
\begin{equation}
\label{eq:poiss2}
\mathfrak{P}_{\M}[\lambda](x) 
= (\lambda\M)!(-\M)^{-\lambda \M} x^{\M(1 - \lambda)} L_{\lambda\M}^{(\M(1 - \lambda))}(x\M)
\end{equation}
which (more obviously) has the atom at $0$.

Again, it should come as no surprise that the resulting polynomial is a Laguerre polynomial, since the root distribution must converge asymptotically to the free Poisson law (or {\em Marchenko-- Pastur}) law, and it is known that the Laguerre polynomials do exactly that.
One does need to check that the correct scaling of the law occurs, which one can do simply by checking that the support of the law matches that of the free Poisson law (the free Poisson law with expectation $\lambda$ has support $[(1 - \sqrt{\lambda})^2, (1 + \sqrt{\lambda})^2]$).
This can be done easily using the following result of Dette and Studden \cite{laguerre_roots}.

\begin{theorem}
\label{thm:lag_roots}
For any $a > 0$, the root distribution of the polynomials
\[
L_{n}^{(a n)}(bnx)
\]
has asymptotic density function
\[
\frac{b}{2\pi x} \sqrt{(r_2 - x)(x - r_1)}
\]
on the interval $[r_1, r_2]$ where 
\[
r_1 
= \frac{1}{b}(2 + a - 2\sqrt{1 + a})
\AND
r_2 
= \frac{1}{b}(2 + a + 2\sqrt{1 + a}).
\]
\end{theorem}
For $\lambda > 1$, we can apply Theorem~\ref{thm:lag_roots} to Equation~(\ref{eq:poiss}) (so $a = \lambda - 1$ and $b=1$) to get a root distribution supported on the interval
\[
[ (1 - \sqrt{\lambda})^2, (1 + \sqrt{\lambda})^2 ].
\]
When $\lambda < 1$, we can instead apply Theorem~\ref{thm:lag_roots} to what is left of Equation~(\ref{eq:poiss2}) 
after removing the atom at $x = 0$ with measure $1 - \lambda$.
This time, we take $a = 1/\lambda$ and $b = 1 / \lambda$ to get a root distribution (again) supported on the interval
\[
[ (1 - \sqrt{\lambda})^2, (1 + \sqrt{\lambda})^2 ].
\]

The limit theorem associated with Poisson distribution is known as the {\em Poisson limit theorem}.
This classical version of this theorem states that if $X_i$ are independent Bernoulli random variables with 
$\mathrm{P}(X_i = 1) = p$, then 
\[
\sum_i X_i \xrightarrow[np \to \lambda]{n \to \infty} \mathrm{Pois}(\lambda)
\]
in distribution.
Here we give a proof in the $\M$-finite case:
\begin{theorem}[Poisson limit theorem]
\label{thm:poisson}
For all $\lambda, \M$ such that $\lambda \M$ is an integer, we have 
\[
[\underbrace{p \sqsum_{\M} \dots \sqsum_{\M} p}_{\lambda \M~\text{times}}] 
= \mathfrak{P}_{\M}[\lambda](x)
\]
where $p(x) = x^{\M - 1} (x - 1)$.
\end{theorem}
\begin{proof}
We start by noting another identity of Laguerre polynomials:
\[
n! x^{\alpha} L_n^{(\alpha)}(x) = (\deriv{x} -1)^n x^{n + \alpha}
\]
which when combined with Equation~(\ref{eq:lag_ident}) (and a change of variables) gives
\begin{equation}
\label{eq:diff}
(\deriv{x} - r)^n x^{n + \alpha} = (n+\alpha)! (-r)^{-\alpha} L_{n + \alpha}^{(-\alpha)}(rx)
\end{equation}
for all constants $r$.

Now we write
\[
p(x) 
= x^{\M} - x^{\M-1} 
= \left(1 - \frac{1}{\M}\deriv{x} \right) \{ x^{\M} \}
\]
so that
\[
[\underbrace{p \sqsum_{\M} \dots \sqsum_{\M} p}_{\lambda \M~\text{times}}]  
= \left(1 - \frac{1}{ks}\deriv{x} \right)^{\lambda \M} \{ x^{\M} \}
= (-\M)^{-\lambda \M} \left(\deriv{x} - \M\right)^{\lambda \M} \{ x^{\M} \}.
\]
So using Equation~(\ref{eq:diff}) with $r = \M$ and $n = \lambda \M$ and $\alpha = (1-\lambda)\M$ gives
\[
[\underbrace{p \sqsum_{\M} \dots \sqsum_{\M} p}_{\lambda \M~\text{times}}] = \M! (-\M)^{-\M} L_{\M}^{((\lambda - 1)\M)}(\M x)
\]
as claimed.
\end{proof}

\subsubsection{Compound Poisson}

The compound Poisson random variable with parameter $\lambda$ and second distribution $\mu$ is characterized by having $i$th cumulant
$\kappa_i = M_i(\mu)$ where $M_i$ is the $i$th moment of $\mu$.
We assume that $\mu$ is distributed uniformly over the roots of a degree $\M$ polynomial $h$.
Hence
\[
\myrtrans{\M}{\mu_A}{s} = \lambda \int \frac{t}{1 - st} \mu(t) \d{t} 
\]
Hence
\[
\ln Q(x) 
= -\lambda \M \iint \frac{t}{1 - xt}\mu(t)\d{x}\d{t} 
= c + \lambda \M \int \ln \left( 1 - xt \right) \mu(t) \d{t}. 
\]
Now if $h(x) = \prod_i (x - r_i)$, then we have
\[
\int \ln \left( 1 - xt \right) \mu(t) \d{t} 
= \sum_i \ln(1 - xr_i) 
= \ln \prod_i (1 - xr_i).
\]
Hence we have
\[
Q\left(\frac{\deriv{x}}{\M} \right) \left\{ x^{\M}\right\} 
\approx \prod_i \left(1 - \frac{r_i}{\M}\deriv{x}\right)^{\lambda \M}\left\{ x^{\M} \right\} 
\approx L_{\M}^{((\lambda-1)\M)} (\M r_1 x) \sqsum_{\M} \dots \sqsum_{\M} L_{\M}^{((\lambda-1)\M)} (\M r_{\M}x)
\]
where $L_{\M}^{(\alpha)}(x)$ is the associated Laguerre polynomial from Equation~(\ref{eq:laguerre}).

\subsection{Restricted Invertibility}

To conclude the applications, we relate the well known concept of {\em restricted invertibility} first proved by Bourgain and Tzafriri in \cite{BT} to the theory developed here.
An argument of this type was first introduced in \cite{ICM}, but many aspects of the proof become much more intuitive in the language of finite free probability.
Theorem~3.1 of \cite{ICM} is the following:
\begin{theorem}\label{thm:ri}
If $v_1, \dots, v_n \in \R^{\M}$ are vectors with 
\[
\sum_{i = 1}^{n} v_i v_i^T
= I,
\]
then for all $k < n$, there exists a set $S \subset [n]$ with $|S| = k$ such that 
\[
\lambda_k\left(\sum_{i \in S} v_i v_i^T\right) 
\geq \left(1 - \sqrt{\frac{k}{\M}}\right)^2 \left(\frac{\M}{n}\right)
\]
\end{theorem}

Those familiar with random matrix theory might recognize the quantity
\[
\left(1 - \sqrt{\frac{k}{\M}}\right)^2 \left(\frac{\M}{n}\right)
\]
as being the lower bound on the spectrum of the Marchenko--Pastur distribution with parameters $\lambda = k/\M$ and intensity $\alpha =\M/n$.
This should not be much of a surprise, as the typical way one would form such a distribution would be to consider the spectrum of 
\[
\sum_{i = 1}^k u_i u_i^T
\]
where the $u_i$ are {\em random} vectors where each coordinate of $u_i$ is an independent Gaussian random variable with variance $\M/n$.
The resulting matrix is known as a {\em Wishart matrix}, and the distribution of the eigenvalues of such matrices was calculated in \cite{MP}.
Using the results of Voiculescu, the addition of these random vectors approaches free convolution (in the asymptotic limit), and so one can instead see this as the sum of $k$ freely independent rank-1 matrices with trace $\M/n$.
That is, the suspected bound in the finite case seems to be governed by the calculation in free probability.

The finite free version of rank-1 matrices with trace $\M/n$ is simply 
\[
p(x) 
= x^{\M} - \left(\frac{\M}{n} \right)x^{\M-1}
\]
and a calculation similar to the one done in Theorem~\ref{thm:poisson} gives 
\[
[\underbrace{p \sqsum_{\M} \dots \sqsum_{\M} p}_{k~\text{times}}] 
= \M! (-n)^{-\M} L_{\M}^{k - \M}(n x)
\]
where we pick $\lambda = k/\M$.
The remainder of the proof then lies in building an ``interlacing family'' as developed in \cite{IF1} that allows one to translate bounds on the roots of this Laguerre polynomial to bounds on the roots of individual polynomial.
One can easily check that such an interlacing family can be built by picking $k$ vectors uniformly at random {\em with replacement}.
That is, if $u_i$ is a random vector that is uniformly distributed over $v_1, \dots, v_n$, then 
\[
\E{}{ \mydet{xI - \sum_{i=0}^k u_i u_i^T }} 
= \M! (-n)^{-\M} L_{\M}^{(k-\M)}(n x)
\]
and the hierarchy of polynomials that one gets by picking each vector one at a time forms such an interlacing family.
We refer the reader to \cite{ICM} for details.

%

\section{Conclusions and Acknowledgements}

The purpose of this paper was to draw the connection between the recent work of the author with Daniel Spielman and Nikhil Srivastava \cite{ICM, conv, IF1, IF2,  IF4} with free probability and more traditional random matrix theory.
It shows how the asymptotic intuitions that have been established in each of these fields can be translated into finite results, both computationally and theoretically.
It gives a unified framework for solving problems using the ``method of interlacing polynomials'' that the author hopes will inspire new and creative uses.

\subsection{Further Research} \label{sec:ff_more}

There are a number of possible interesting directions.
Ongoing work between the author with Dimitri Shlyakhtenko and Nikhil Srivastava extends the concepts of free entropy and free Fisher information to the finite setting.
Finite free probability also seems to have an interesting combinatorial interpretation that draws from both the classical (all partitions) and the free (noncrossing partitions) interpretations.
Extending the results here to asymmetric matrices and to other concepts in free probability, like freeness with amalgamation \cite{amalgamation}, could also lead to interesting new applications.
Of particular interest would be an extension to bivariate polynomials, which is possibly related to the concept of {\em second order freeness} introduced by Mingo and Speicher \cite{mingo2006second}.
Such an extension would have the potential to place more advanced results such as \cite{TSP, IF2} under a similar umbrella, which would be a notable advance in the understanding of how such results fit into the free probability framework. 

It would be interesting to see direct applications to random matrix theory, particularly in the realm of universality. 
Since such results typically use the relationship between polynomials and random matrices (and free probability) in a somewhat ad-hoc way, one might hope that a theory connecting the two would be useful.
The concept of finite free entropy seems to be directly related to such pursuits, since its manifestation employs a {\em logarithmic potential}, a topic that has led to a number of results in the asymptotic root distributions of polynomials \cite{saff}.
In the reverse, random matrix theory could be useful in establishing concentration of measure results in this paradigm --- something typically necessary for widespread applications.

Lastly, we mention possible implications in quantum information theory.
In fact, results of this sort have already been applied in such a context: when $A$ and $B$ are in finite free position and $B$ has rank $1$, the (single) nontrivial eigenvector of $B$ coincides with the so-called {\em trace vector} that was introduced by Murray and von Neumann in their initial work in $C^*$ algebras \cite{rings}.
Trace vectors have been used to obtain results in private quantum channels \cite{trace_vectors}, and so the fact that finite freeness is a vast generalization of this concept gives promise to the possibility of further applications.

\subsection{Acknowledgements}

This work is an accumulation of ideas that has formed through an incredible partnership with Daniel Spielman and Nikhil Srivastava.
The author wishes to recognize two conferences: the ``Beyond Kadison--Singer: paving and consequences'' workshop at AIM, and the ``Hot Topics: Kadison--Singer, Interlacing Polynomials, and Beyond'' workshop held at MSRI.
Both were influential in giving the author the foothold in the numerous fields that this work used as inspiration, and such interactions would not have happened without such support.
The author was also helped enormously by participating in the free probability workshop hosted by Dan Voiculescu and Dmitri Shlyakhtenko.   
In addition to thanking the many people people who have given insights into 
this topic, the author would like to specifically mention Ken Dykema and Dmitri 
Shlyakhtenko as both were instrumental at different times in steering the 
author through areas in which the author is still a novice.  

\bibliographystyle{abbrv}
\bibliography{ffbib}


\end{document}